\documentclass{aims}
\usepackage{amsmath,amssymb,bm}
\usepackage{multirow,booktabs,lineno}
\usepackage{paralist}
\usepackage{xcolor}
\usepackage{color}
\usepackage{graphics} 
\usepackage{epsfig} 
\usepackage{graphicx,subfigure}
\usepackage[colorlinks=true]{hyperref}
\hypersetup{urlcolor=blue, citecolor=red}
\usepackage{multirow}
\def\currentvolume{X}
\def\currentissue{X}
\def\currentyear{2024}
\def\currentmonth{1}
\def\ppages{XX--XX}

\newtheorem{theorem}{Theorem}[section]

\newtheorem{lemma}[theorem]{Lemma}

\theoremstyle{definition}
\newtheorem{definition}[theorem]{Definition}

\title[PIGNN] 
{Combining physics-informed graph neural network and finite difference for solving forward and inverse spatiotemporal PDEs}

\thanks{* Corresponding author: wangly\_xjtu@163.com }

\begin{document}

\maketitle

\centerline{\scshape Hao Zhang, Longxiang Jiang, Xinkun Chu, Yong Wen,}
\footnote{Emails: linusec@163.com,\,\, whutjlx@163.com,\,\, neocosmos@163.com,

\qquad\qquad wenyong119321@163.com,\,\, m13368289607\_1@163.com,\,\, xiao\_yonghao@163.com.}
\bigskip
\centerline{\scshape Luxiong Li, Yonghao Xiao, Liyuan Wang$^*$}

{\small \centerline{Institute of Computer Application, China Academy of Engineering Physics, }
\centerline{Mianyang, Sichuan, 621000, P. R. China.}}
\medskip

\begin{abstract}
The great success of Physics-Informed Neural Networks (PINN) in solving partial differential equations (PDEs) has significantly advanced our simulation and understanding of complex physical systems in science and engineering. However, many PINN-like methods are poorly scalable and are limited to in-sample scenarios. To address these challenges, this work proposes a novel discrete approach termed Physics-Informed Graph Neural Network (PIGNN) to solve forward and inverse nonlinear PDEs. In particular, our approach seamlessly integrates the strength of graph neural networks (GNN), physical equations and finite difference to approximate solutions of physical systems. Our approach is compared with the PINN baseline on three well-known nonlinear PDEs (heat, Burgers and FitzHugh-Nagumo). We demonstrate the excellent performance of the proposed method to work with irregular meshes, longer time steps, arbitrary spatial resolutions, varying initial conditions (ICs) and boundary conditions (BCs) by conducting extensive numerical experiments. Numerical results also illustrate the superiority of our approach in terms of accuracy, time extrapolability, generalizability and scalability. The main advantage of our approach is that models trained in small domains with simple settings have excellent fitting capabilities and can be directly applied to more complex situations in large domains.
\end{abstract}

\bigskip

\bigskip

\section{Introduction}\label{sec1}
\noindent{\textbf{\\}}

PDEs are indispensable mathematical tools to describe the behavior of complex dynamical systems under varying conditions in multifarious science and engineering disciplines, such as thermodynamics, fluid mechanics, biochemistry, electronics, and materialogy. By studying PDEs, we can learn the evolutionary patterns of corresponding systems, in the meantime, understand the influence of key parameters of PDEs on the evolution process.

In scientific computing, there are three main issues related to PDE research: forward problem (solving PDEs), inverse problem and equation discovering problem. In forward problems, governing equations are known, we need to find analytical or numerical solutions using conventional numerical methods or other surrogate methods. The inverse problem represents using observations to infer unknown parameters of the physical system, i.e., information of the governing equation is partially known. While discovering governing equation refers to the situation that the dynamical system has no analytic expressions, fortunately in these cases, observation data is abundant, and one needs to determine the causal factors or the potential governing equation from these observations. This paper focusses on the first two problems.

Except for a few simple PDEs that have analytical solutions, most of PDE problems need to be solved numerically. Decades of researches and efforts have been devoted to the investigation and development of numerical methods. At present, the mainstream numerical approaches include finite difference, finite element and finite volume method. Although the increasing computational power and capacity enable the computation of complex problems, numerical methods still suffer from enormous computational costs: long time consumption and resource burden. This issue becomes even more serious when solving complex inverse problems.

Recently machine learning (ML), especially deep learning (DL) technology, emerges as a promising direction for solving scientific problems, such as exploring physical systems that are not yet fully understood and accelerating simulation process, due to its powerful feature extraction ability and computation efficiency. Nevertheless, because of their black-box nature, applying DL model directly is insufficient to solve complex scientific problems. The research community is beginning to integrate the scientific knowledge and DL models for the purpose of alleviating the shortage of training data, increasing models' generalizability and ensuring the physical plausibility of results.

In the past few years, integrating scientific knowledge with DL models has been a research boom. The primitive idea can be traced back to Raissi et al. \cite{Raissi2019}. They proposed a novel framework of neural network named PINN, which successfully integrates the underlying physical law described by PDEs with fully connected (FC) networks. This work lays a solid foundation for solving both forward and inverse PDEs and has inspired plenty of extended studies. The general principle of PINN is to take the physical information as constraint and incorporate it into the loss function to steer the deep model to output physically consistent predictions. The popularity of PINN is mainly attributed to three aspects: the flexibility of meshless, the simplicity of the neural network, and the automatic differentiation technique. PINN achieves great success in many applications fields, including fluid mechanics \cite{Cai2022, Wong2021, He2020, Mao2020}, medical diagnosis \cite{Arzani2021, Sahli2020}, materials science \cite{Yin2021, ZhangYin2020}, engineering mechanics \cite{Haghighat2021, ZhangLiu2020}. Despite the success of PINN, it still suffers from some limitations. On the one hand, PINN lacks the ability to generalize to out-of-sample scenarios, such as varying time scales, spatial resolutions and computational domains. On the other hand, PINN requires a large number of collocation points to calculate PDE residuals, which can lead to a huge training cost, especially in high-dimensional space. Based on PINN, a lot of extension works have been done to improve the performance of PINN, such as synthesizing with numerical methods \cite{Kharazmi2019}, defining problem-specific activation functions \cite{Sitzmann2020}, optimizing the sampling of collocation points \cite{Mao2020, Wight2020, Nabian2021}, and balancing residuals with adaptive weights \cite{Kharazmi2019, Wight2020, Pang2019, McClenny2020, Xiang2021}, configuring special neural network frameworks \cite{Gao2021, Rodriguez2021, Ren2022}. 

In addition to physics-based models, data-driven methods \cite{Zhang2019, Chen2021, Meng2022} also seek nonlinear mappings from parametric DNNs to numerical solutions. In 2020, a Google research team \cite{Sanchez2020} applied GNN to particle simulation, where dynamics between particles are simulated through a message passing mechanism. Although this method achieves the best simulation accuracy at the time, its computational efficiency is low due to its neighborhood calculation method. Later, this team \cite{Pfaff2020} further expanded their work by applying GNNs to mesh-based simulations and proposed the MeshGraphNets network. By evaluating a wide range of physical systems, this method was shown to predict dynamics of underlying systems accurately and efficiently, and was capable of learning resolution adaptive dynamics and scalable to more complex spaces.

As mentioned above, based on existing advances in the physics-informed DL community, the goal of this paper is to construct a surrogate model for solving spatiotemporal PDEs in irregular domains. This model learns mesh-based solutions of PDEs with a few or no pre-computed labeled data by integrating physical information into GNN with acceptable accuracy and generalizability while substantially reducing the computational cost. We are interested in dealing with unstructured mechanisms rather than regular grids, as in practice, the computational domain of complex physical systems is usually irregular. Mesh is a powerful tool for appropriately representing irregular geometry, and employing mesh representation to solve PDEs supports adaptive resolutions and achieves a favorable tradeoff between accuracy and efficiency. Based on the mesh generated by numerical methods, by establishing an one-to-one mapping between mesh points and graph nodes, we can describe unstructured mesh data with graphs. Therefore, we naturally use GNNs instead of other neural networks. GNN can not only model the structure between nodes, but also perform local information mining and relational reasoning based on the message passing mechanisms, which implicitly introduces inductive biases that is conductive to the learning process.

In a word, we propose a novel PIGNN framework, combined with a finite-difference-based method for computing differential operators on the graph, so as to accurately and effectively solve the forward and inverse PDEs in mesh space. This work provides an important step towards advancing computation of solving PDEs. Compared to the existing literature, our main contributions are as follows:

1) We introduce a novel PIGNN framework to efficiently solve the forward and inverse PDEs. For the forward problem, the computation of physics-informed loss requires no labeled data. While for the inverse problem, regularizations constrained by underlying PDEs and few labeled data narrow the search space of parameters, which makes it possible for the model to rapidly learn unknown parameters and solutions with ensured physical consistency during optimization.

2) We propose a finite-difference-based method combined with the ordinary least square technique to approximate the differential operators defined on the graph. At the same time, we theoretically discuss the approximation error of differential operators.

3) The proposed PIGNN is directly constructed from the unstructured mesh, providing the ability to handle irregular domains. Capitalizing on the powerful mesh representation and learning with certain resolution, our method can generalize to scenarios with refined or coarse resolution (i.e., other discretization schemes) not seen during training.

4) We present general rules of selecting appropriate features for nodes and edges, which are found to be an important factor in enabling the model to better generalize to out-of-sample scenarios.

5) A series of numerical experiments are conducted on several typical PDEs to quantitatively evaluate our framework. Numerical results show that our method outperforms the baseline model in terms of accuracy, time extrapolation, generalization and scalability. Most importantly, our model trained in small domains with simple settings has excellent fitting ability, and can be directly applied to complex settings with large domains, which further sheds lights on the superiority of the proposed PIGNN framework.

6) Compared with \cite{Jiang2023}, the main difference of this paper is that we apply PIGNN to inverse problems and extend PIGNN to distributed multi-GPU systems to solve large-scale problems, so as to show the scalability of PIGNN. Meantime, the theoretical error estimation of differential operators is presented.

The rest of this paper is organized as follows. In Section 2, we elaborate on the PIGNN framework, including the problem setup, model architecture, physics-informed loss function and general rules of selecting features. Section 3 illustrates some computational issues defined on the graph. We first present the estimation method for differential operators on the mesh, then state the interpolation method involved in the inverse problem. Section 4 provides some numerical experiments to illustrate the outstanding performance of our PIGNN. Finally, Section 5 concludes the paper.\vspace{2mm}

\section{PIGNN Framework}\label{sec2}
\noindent{\textbf{\\}}
\subsection{Problem setup}\label{subsec2.1}
\noindent{\textbf{\\}}

We consider a dynamical system in a bounded domain $\Omega\subset{\mathbb{R}}^{d}$ governed by a system of nonlinear PDEs of the following general form,
\begin{eqnarray} \label{Eq2.1}
\frac{\partial\, \bm{{\rm u}}(\bm{{\rm x}},t)}{\partial t}
+ {\mathcal F}\big[\bm{{\rm u}}, \nabla_{\bm{{\rm x}}}\bm{{\rm u}},
    \Delta_{\bm{{\rm x}}}\bm{{\rm u}}, \cdots; \bm{\lambda}\big]
= \bm{{\rm f}}(\bm{{\rm x}},t),
&&
    \bm{{\rm x}}\in\Omega, \,t\in{\mathbb T}=[0,T],   \label{Eq2.1} \\
{\mathcal B}\big(\bm{{\rm u}}(\bm{{\rm x}},t)\big) = b(\bm{{\rm x}},t),
&&
    \bm{{\rm x}}\in\partial \Omega, \,t\in{\mathbb T},  \label{Eq2.2} \\
{\mathcal I}\big(\bm{{\rm u}}(\bm{{\rm x}},0) \big) = i(\bm{{\rm x}}),
&&
\bm{{\rm x}}\in \Omega, \,t=0.\label{Eq2.3}
\end{eqnarray}
\noindent where the system state $\bm{{\rm u}}(\bm{{\rm x}},t)\in{\mathbb{R}}^{n}$ denotes the solution of PDEs evolved over time horizon $t\in{\mathbb T}$ and spatial locations $\bm{{\rm x}}\in\Omega$, $T$ is the terminal time, $\frac{\partial \bm{{\rm u}}}{\partial t}$ is the first-order time derivative term, ${\mathcal{F}}[\cdot\,; \bm{\lambda}]$ is a complex nonlinear functional of state variable $\bm{{\rm u}}(\bm{{\rm x}},t)$ and its spatial derivatives, parameterized by $\bm{\lambda}$. Here, $\nabla$ and $\Delta$ denote the gradient and Laplace operators, respectively. $\bm{{\rm f}}(\bm{{\rm x}},t)$ is the source term ($\bm{{\rm f}}=\bm{0}$ represents no extra source input to the system). In addition, the system of PDEs is subjected to BCs and ICs, defined by Eq. \eqref{Eq2.2}-\eqref{Eq2.3}, where $\partial \Omega$ denotes the boundary of spatial domain. Such PDEs describe a wide range of propagative physical systems, including heat dissipation, disease progression, cellular kinetics, fluid dynamics and so on.

In this paper, our objective is to propose an innovative PIGNN paradigm to establish a solution method for the above PDE system in both forward and inverse settings. In the forward setting, given specific ICs, BCs and PDE parameter $\bm{\lambda}$, we aim to find solutions for PDEs. The physics-informed loss function is constructed based on the residuals of PDEs, and the overall learning process is unsupervised. While for the inverse problem, ICs and BCs are given, $\bm{\lambda}$ is unknown, but a few observations of the system state are available. What we need to do is to determine the system response as well as identify the unknown parameter $\bm{\lambda}$ that best fits the observations. The working mechanism of PIGNN is shown in Fig. \ref{fig1}. In general, we can interpret PIGNN as a neural solver for PDEs. The prior knowledge of physics is incorporated into PIGNN through specially designed loss functions to capture the dynamic patterns.

\begin{figure}[!htbp]
\begin{center}
\scalebox{0.49}[0.5]{\includegraphics{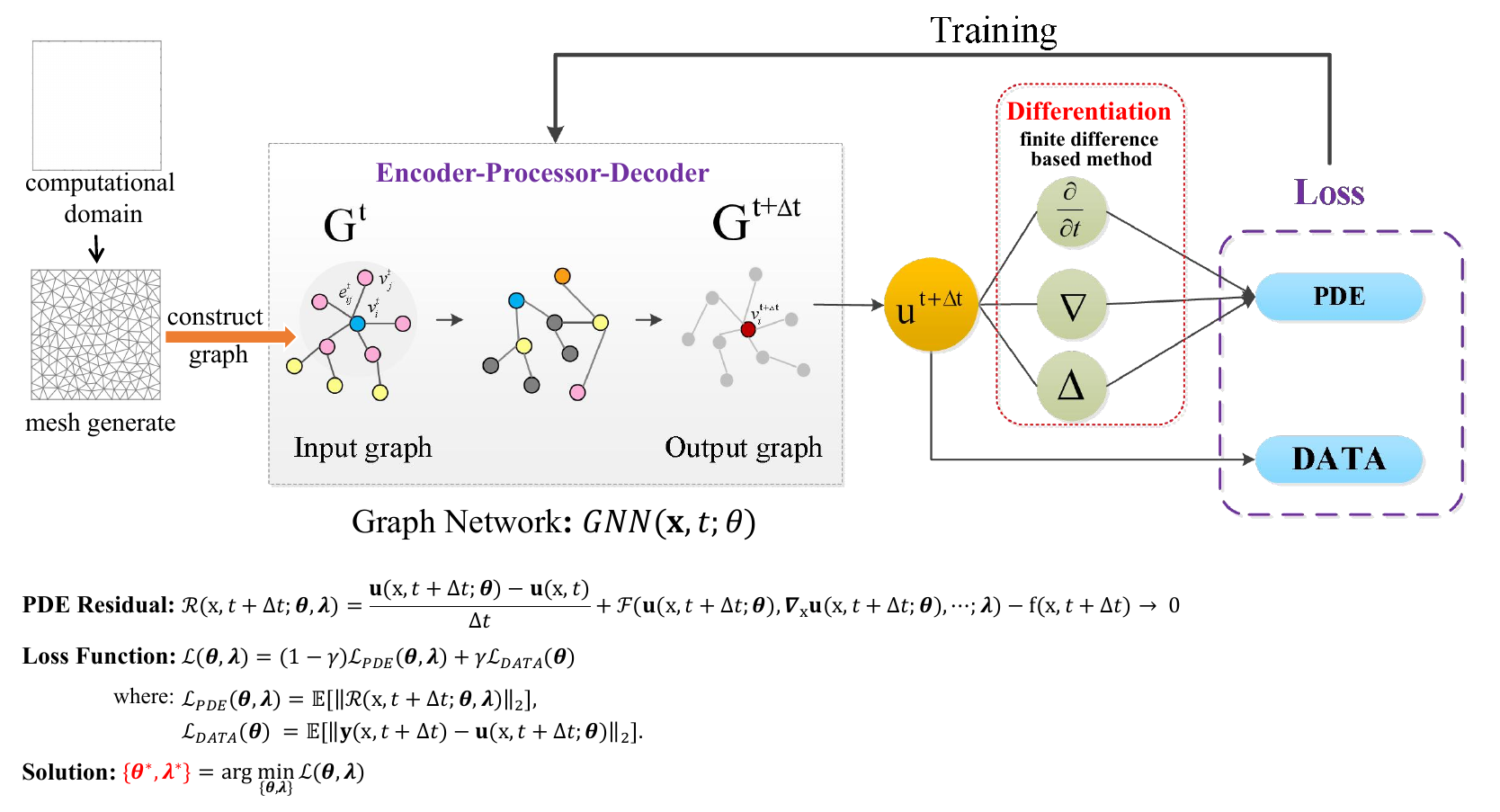}}
\caption{\footnotesize{The schematic diagram of PIGNN for solving forward and inverse problems.
Unstructured meshes are generated using numerical methods. The Encoder component of the network constructs the input graph from mesh data, Processor updates and transmits the information, and Decoder outputs PDE solutions. By applying the finite-difference-based method, we compute all differential terms and incorporate prior knowledge of PDEs into the network via the specific PDE loss.
}}
\label{fig1}
\end{center}
\end{figure}

Simulation of PDE dynamics relies on appropriate spatial-temporal discretizations. Our first step is to discretize the computational domain $\Omega$ by constructing a simulation mesh on the domain surface, which can be implemented by mature mesh generation algorithms, such as Delaunay triangulation, advancing front method, and mapping method. At time $t$, let $\mathcal M^t=(\mathcal V, \mathcal E^{M})$ be an unstructured mesh with mesh nodes $\mathcal V$ connected by mesh edges $\mathcal E^{M}$. $N=\|\mathcal V\|$ denotes the number of mesh nodes. Each node $i\in \mathcal V$ is associated with a coordinate $\bm{{\rm x}}_i$ in the mesh space and a feature vector $\bm{f}_i$ describing the system state. By constructing an one-to-one mapping between mesh nodes and graph nodes, the unstructured mesh can be naturally represented as a graph. Let $G^{t}=(V, E^{G})$ be the undirected graph, where $V=\{\bm{{\rm x}}_i\}_{i=1}^{N}$ is the set of graph nodes, and the coordinate of the $i$-th node on the graph is still represented by its mesh space coordinate because of the one-to-one mapping. $E^{G}=\{\bm {e}_{ij}\}_{i,j\in V}$ denotes the set of graph edges, where the directional edge $\bm{e}_{ij}$ stands for the connection from node $i$ (sender) to node $j$ (receiver). Each node and edge is associated with a node feature vector and an edge feature vector, separately. By abuse of notation, we use the same symbols for edge features $E^G = \{\bm{e}_{ij}\}_{i,j\in V}$ and node features $V = \{\bm{v}_i\}_{i=1}^N$ on the graph. $\bm{e}_{ij}=\{\bm{{\rm x}}_{ij}, \|\bm{{\rm x}}_{ij}\|\}$ represents the edge feature vector, which in our setting contains the relative displacement vector between node $i$ and $j$: $\bm{{\rm x}}_{ij} = \bm{{\rm x}}_i-\bm{{\rm x}}_j$ and its Euclidean distance $\|\bm{{\rm x}}_{ij}\|$. $\bm{v}_i$ denotes the feature vector of the $i$-th node, including the learned node solution $\bm{{\rm u}}(\bm{{\rm x}},t)$, a 2-dimensional one-hot vector indicating the node type, and other physical variables providing additional spatiotemporal information. So far, solutions of the PDE system have been successfully discretized by these graph nodes. Let $\bm {{\rm u}}(\bm{{\rm x}}_i,t)$ be the $i$-th node solution of PDEs at time $t$. Our goal now is to find node solutions to the PDE system Eq. \eqref{Eq2.1}-\eqref{Eq2.3} on the graph for a time range after a certain initial time.

\subsection{Model architecture}\label{subsec2.2}
\noindent{\textbf{\\}}

Note that the current time $t$ is equal to the current time step multiplied by the time interval. We use the term time step to represent the time variable in the rest of this article for convenience. Assuming that $F: {\mathbb{R}}^{n}\rightarrow {\mathbb{R}}^{n}$ models the dynamics of solutions from the current time step $t$ to the next time step $t+\Delta t$, where $\Delta t$ is the time interval. In other words, at each time step, mesh-based solutions of system Eq. \eqref{Eq2.1}-\eqref{Eq2.3} are computed iteratively in a rollout manner by, $\bm {{\rm u}}(\bm{{\rm x}}_i, t+\Delta t)=F\big(\bm {{\rm u}}(\bm{{\rm x}}_i, t)\big)$. Our task then is, given the graph at time step $t: G^{t}$, to learn a model that determines dynamic quantities of the graph at the next time step $t+\Delta t$. We consider a general type of GNN proposed by Gilmer et al. \cite{Gilmer2017} to model variations of solutions from the previous time step to the consequent time step, namely, $F\big(\bm {{\rm u}}(\bm{{\rm x}}_i, t)\big)=\bm {{\rm u}}(\bm{{\rm x}}_i, t)+ GNN(\cdot; {\bm\theta})$, where ${\bm\theta}$ denotes the parameters of GNN. Our PIGNN adopts the same Encoder-Processor-Decoder architecture as that in Sanchez-Gonzalez et al. \cite{Sanchez2020, Pfaff2020}, which is shown in Fig. \ref{fig2}.
\begin{figure}[!htbp] 
\begin{center}
\scalebox{0.78}[0.78]{\includegraphics{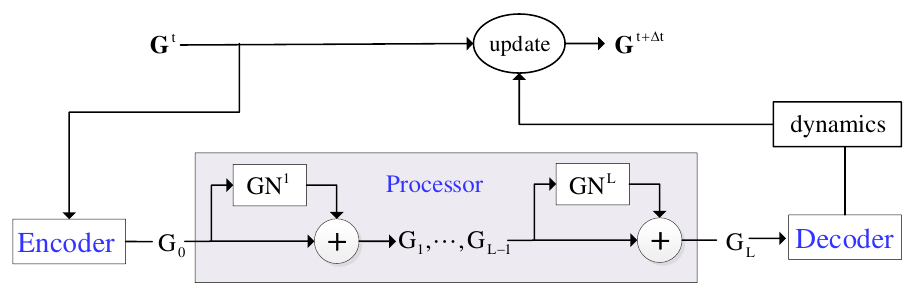}}
\caption{\footnotesize{The network architecture of PIGNN.}}
\label{fig2}
\end{center}
\end{figure}

\begin{enumerate}
	\item
	ENCODER.
\medskip

	Encoder converts the unstructured mesh at time step $t$ into a latent graph $G_0^t = {\rm ENCODER}\big({\rm u}(\bm{{\rm x}},t)\big)$. To achieve spatial equivalence, mesh nodes become graph nodes $V=\{\bm{{\rm x}}_i\}_{i=1}^{N}$, and mesh edges become bidirectional edges $E^G=\{\bm{e}_{ij}\}$ of the graph. As mentioned above, each node and edge is associated with a node feature $\bm{v}_i$ and an edge feature $\bm{e}_{ij}$, and their settings are the same as before. The node/edge features are concatenated as input node/edge features, which are then transformed into high-level features using several multilayer perceptrons (MLPs) to obtain sufficient expressive power.
\medskip
	
	\item
	PROCESSOR.
\medskip

	At time step $t$, the processor consists of a stack of message passing blocks named GN (Graph Network) blocks, which models interactions among nodes and propagates information by running $L$ steps to generate a sequence of updated latent graph $\bm{G}^t=(G^t_1, \cdots, G^t_L)$, where each block $l\in\{1,\cdots,L\}$ is computed based on the output of the previous block, i.e., $G^t_l=GN(G^t_{l-1})$. In a GN block, computations proceed from edges to nodes, which contains two update functions and an aggregation function, following three procedures: updating the edge feature for each edge, then aggregating the updated edge features and projecting them to node features, and finally updating the node feature using the aggregated edges information. Mathematically,
	\begin{eqnarray}
		\text{updating edge feature}: \qquad
		\bm{\tilde{e}}_{ij,(l)} &=&
		\phi^{e}\big(\bm{e}_{ij,(l)}, \bm{v}_{i,(l)}, \bm{v}_{j,(l)}\big), \\
		\text{aggregating edge features per node}: \qquad
		\bm{\bar{e}}_{i,(l)} &=&
		\bigoplus_{j\in {\mathcal N}(i)}  \bm{\tilde{e}}_{ij,(l)} , \\
		\text{updating node feature}: \qquad
		\bm{\tilde{v}}_{i,(l)} &=&
		\phi^{v}\big(\bm{\bar{e}}_{i,(l)}, \bm{v}_{i,(l)}\big),
	\end{eqnarray}
where $\phi^{e}$ and $\phi^{v}$ are edge and node update functions respectively. These functions are mapped across all edges/nodes to compute per-edge/per-node update. In our implementation, MLPs are chosen as internal edge and node update functions because of its universal approximation property. The subscript $(l)$ represents the block index. $\bigoplus$ denotes a permutation invariant aggregate function (such as sum, mean, maximum), which takes a set of edge updates as input and outputs a single element by projecting it onto the $i$-th node to represent the aggregated information. $\mathcal{N}(i)$ denotes indices of neighbours referring to a set of adjacent nodes that are connected to node $i$ via edges. $G_0^t$ is provided as the input of the processor's first GN block, we initialize the node features of graph $G_0^t$ to the current PDE solutions, i.e., $\bm{v}_{i,(0)} = \bm{{\rm u}}(\bm{{\rm x}}_i, t)$. Based on the above illustration, we see that the processor computes interactions among related nodes and edges, propagates the information via message passing steps and finally returns a well updated graph $G^t_L$, that is, $G^t_L = PROCESSOR\big(G^t_0\big)$.
\medskip

	\item
	DECODER.
\medskip

	Decoder extracts dynamics from the node features $\bm{v}_{i,(L)}$ of the final latent graph $G^t_L$, namely, $\bm{h}_{i,(L)}=\text{DECODER}\big(\bm{v}_{i,(L)}\big)$. As before, we use a MLP as the learned function for the decoder. To compute the next time step solutions, we integrate the output $\bm{h}_{i,(L)}$ using a forward-Euler integrator with time interval $\Delta t=1$: $\bm{{\rm u}}(\bm{{\rm x}}_i, t+\Delta t)=\bm{{\rm u}}(\bm{{\rm x}}_i, t) + \bm{h}_{i,(L)}$. Therefore, decoder's output represents updated quantities of the PDE system for two consecutive time steps, which is semantically meaningful to the update process.

\end{enumerate}

\subsection{Physics-informed loss function}\label{subsec2.3}
\noindent{\textbf{\\}}

To capture generalizable dynamic patterns consistent with underlying physical laws, one of the common techniques is incorporating physical constraints into the loss function of the model. The basic concept is to tune the model trainable parameters and unknown PDE coefficients so that the network can fit the observations while satisfying constraints defined by the underlying PDEs. In order to propose a well-posed optimization problem to improve solution accuracy and facilitate convergence, BCs/ICs are strictly enforced during network training. That is, IC is assigned to the nodes of the graph at the initial time step, and BC is assigned to the predicted solutions before computing PDE loss. The motivation for hard embedding BCs/ICs is to strictly incorporate all known information into the network to guide the network to quickly output physical plausible predictions.

Therefore, for the forward problem, our physics-informed loss function is simply the residual physics loss, which is given by
\begin{equation}
	\mathcal{L}(\bm{\theta}; \mathcal{D}_p)=
	\mathcal{L}_{PDE}(\bm{\theta};\mathcal{D}_p),
\end{equation}
where $\bm{\theta}$ represents the trainable parameters of the model. $\mathcal{D}_p = \{\bm{{\rm x}}_i^p, t_p\}_{i=1}^{N_p}$ is collocation points set for $p=1,\cdots,N_{t_p}$.

For the inverse problem, the unknown parameter $\bm{\lambda}$ can be inferred by assimilating the observation data in a soft manner, where the above physics-informed loss is augmented by a data loss. That is, the following loss function is formulated:
\begin{equation}
	\mathcal{L}(\bm{\theta}, \bm{\lambda}; \mathcal{D}_p, \mathcal{D}_d)=
	(1-\gamma)\mathcal{L}_{PDE}(\bm{\theta}, \bm{\lambda}; \mathcal{D}_p)
    + \gamma \mathcal{L}_{Data}(\bm{\theta}; \mathcal{D}_d),
\end{equation}
where $\gamma$ is a weighted parameter of the data loss, $\mathcal{D}_d =\{ \bm{{\rm y}}(\bm{{\rm x}}_i^{d}, t_{d})\}_{i=1}^{N_d}$ is the observation data set. It is worth noting that the observed time point $t_d\, (d=1,\cdots,N_{t_d})$ may not equal the predicted time point $t_p$.

In the discrete-time setting, $\{\bm{{\rm x}}_i^p,t_p\}$ selected from the interior mesh nodes and the corresponding time points acts as collocation points for calculating the PDE loss term $\mathcal{L}_{PDE}(\bm{\theta},\bm{\lambda}; \mathcal{D}_p)$. Additional data loss will be assessed when the inverse problem is considered. The observation data $\{\bm{{\rm y}}(\bm{{\rm x}}_i^{d}, t_d)\}$ and the collocation points will be superimposed together to calculate the data loss and the PDE loss. $N_p$ and $N_d$ are the number of collocation points and observations for network training, respectively.

Concretely, the PDE loss aims to ensure consistency with physical laws, which is defined as the mean of squared residuals given by
\begin{equation*}
\mathcal{L}_{PDE}(\bm{\theta},\bm{\lambda}; \mathcal{D}_p) =
{\mathbb E}\big[\big\|\mathcal{R}(\bm{{\rm x}}^{p},t_p;\bm{\theta},\bm{\lambda})\big\|_2^2\big],
\end{equation*}
where $\|\cdot\|_2$ represents the $l_2$ norm, and the PDE residual is
\begin{equation*}
\mathcal{R}(\bm{{\rm x}}^{p},t_p;\bm{\theta},\bm{\lambda}) =
\frac{{\partial \bm{{\rm u}}(\bm{{\rm x}}^{p},t_p; \bm{\theta})}}{\partial t}
+ {\mathcal F}\big[\bm{{\rm u}}, \nabla_{\bm{{\rm x}}}\bm{{\rm u}},
\Delta_{\bm{{\rm x}}}\bm{{\rm u}}, \cdots; \bm{\lambda}\big]
- \bm{{\rm f}}(\bm{{\rm x}}^{p},t_p).	
\end{equation*}

The data loss measures the MSE between observations $\bm{{\rm y}}(\bm{{\rm x}}^{d},t_d)$ and predictions $\bm{{\rm u}}_{pred}(\bm{{\rm x}}^{d},t_d; \bm{\theta})$ of PIGNN,
\begin{equation}
	\mathcal{L}_{Data}(\bm{\theta}; \mathcal{D}_d) =
	{\mathbb E}\big[ \big\| \bm{{\rm y}}(\bm{{\rm x}}^{d},t_d)
    - \bm{{\rm u}}_{pred}(\bm{{\rm x}}^{d},t_d; \bm{\theta}) \big\|_{2}^2\big]. \nonumber
\end{equation}

With the above spatiotemporal discretization scheme, $\mathcal{L}(\bm{\theta},\bm{\lambda})$ can be equivalently expressed as,
\begin{eqnarray}\label{Eq2.8}
\mathcal{L}(\bm{\theta},\bm{\lambda};\mathcal{D}_p, \mathcal{D}_d) &=&
\frac{1-\gamma}{N_p N_{t_p}}\sum_{p=1}^{N_{t_p}}\sum_{i=1}^{N_p}
\big\| \mathcal{R}(\bm{{\rm x}}_i^{p},t_p;\bm{\theta},\bm{\lambda})\big\|_{2}^2  \\
&&
+ \frac{\gamma }{N_dN_{t_d}}\sum_{d=1}^{N_{t_d}}\sum_{i=1}^{N_d}
\big\| \bm{{\rm y}}(\bm{{\rm x}}_i^d,t_d)
- \bm{{\rm u}}_{pred}(\bm{{\rm x}}_i^{d},t_d; \bm{\theta}) \big\|_{2}^2.  \nonumber
\end{eqnarray}

At time step $t$, the predicted solutions of the next time step $t+\Delta t$ are separated from the calculation graph and then flowed into the network to calculate solutions for the step after next $t+2\Delta t$. The network repeats this procedure to calculate predictions for all training time steps, and accumulates the losses of all steps to update the network parameters in one go. Repeat the above process until a specified number of training iterations is reached or the network converges. A well optimized PIGNN model with optimal parameters $\bm \theta^*$ can not only approximate the PDE solutions, but also identify the unknown parameter of PDEs $\bm{\lambda}^*$, namely, $$\{\bm{\theta}^*,\bm{\lambda}^*\} = \arg\min_{\{\bm{\theta},\bm{\lambda}\}} \mathcal{L}(\bm{\theta},\bm{\lambda};\mathcal{D}_p, \mathcal{D}_d).$$

\subsection{Rules of selecting features}\label{subsec2.4}
\noindent{\textbf{\\}}

According to the form of Eq. \eqref{Eq2.1}, if a system has no additional source input, i.e., $\bm{{\rm f}}(\bm{{\rm x}},t)$ equals 0, then node features contain the learned node solution at time step $t:\bm{{\rm u}}(\bm{{\rm x}},t)$, and a node type vector indicating whether it is an interior node or a boundary node. While edge features consist of the relative displacement of two adjacent nodes $i$ and $j$ and its Euclidean distance. Namely,

node features: $\bm{v_i}=\{\bm{{\rm u}}(\bm{{\rm x}},t), \rm{\,node \_type}\}$,

edge features: $\bm{e_{ij}}= \{\bm{{\rm x}}_{ij}, \,\|\bm{{\rm x}}_{ij}\|\}$.
\smallskip

However, if the source term $\bm{{\rm f}}(\bm{{\rm x}},t)$ is nonzero, it does affect the evolution of the system, so we add the source term that does not contain solution functions to node features to capture its variations in both spatial and time space. Researchers know the influence and importance of the external source on the system. We need to embed such expert knowledge into our model to enable the network to learn these key dynamic patterns, thereby enhancing the generalizability of the model. The edge features remain the same as above. In this case, the features of nodes and edges are as follows:

node features: $\bm{v_i}=\{\bm{{\rm u}}(\bm{{\rm x}},t), \rm{\,node \_type}, \bm{{\rm f}}(\bm{{\rm x}},t)\}$,

edge features: $\bm{e_{ij}}= \{\bm{{\rm x}}_{ij}, \,\|\bm{{\rm x}}_{ij}\|\}$.
\smallskip

\section{Some computational issues on the graph} \label{sec3}
\noindent{\textbf{\\}}

In this section, we will discuss some computational problems on the graph. Specifically, we discuss methods for computing discrete differential operators on the graph, and the interpolation method adopted in the inverse problem.

\subsection{Differential operators on the graph}
\noindent{\textbf{\\}}

We use finite difference method to compute discrete differential operators (time and spatial derivatives) on a graph.

\subsubsection{Time derivative}
\noindent{\textbf{\\}}

Recall that PIGNN iteratively generates solutions of the next time step $t+\Delta t$ according to
solutions of the current time step $t$. By utilizing solutions of two consecutive time steps, we can estimate the first order time derivative term in Eq. \eqref{Eq2.1}. Under the Euler setting, either explicit, implicit or semi-implicit methods can be used to calculate the time derivative. The explicit method usually contains one or several forward computation steps. Implicit or semi-implicit method, such as Crank-Nicolson's scheme, requires a fixed point iterative solver. Numerical results show that, compared to explicit methods, implicit or semi-implicit methods provide some stability for the outputs in admissible time steps, albeit at the expense of increased computation cost. In this paper, we choose the implicit update rule to calculate the time derivative.

To sum up, due to its simplicity and stability, we use the implicit/backward Euler difference quotient to approximate the time-derivative, and the implicit/backward Euler discretization of Eq. \eqref{Eq2.1} is,
\begin{equation}\label{Eq3.1}
\frac{\bm{{\rm u}}(\bm{{\rm x}},t+\Delta t)-\bm{{\rm u}}(\bm{{\rm x}},t)}{\Delta t}
+ {\mathcal F}\big[\bm{{\rm u}}(\bm{{\rm x}},t+\Delta t),
\nabla_{\bm{{\rm x}}}\bm{{\rm u}}(\bm{{\rm x}},t+\Delta t), \cdots; \bm{\lambda}\big]
= \bm{{\rm f}}(\bm{{\rm x}},t+\Delta t).
\end{equation}

\subsubsection{Gradient operator}
\noindent{\textbf{\\}}

In PINN, spatial derivatives in the PDE loss are calculated by the automatic differentiation technique, which utilizes the accumulated values during code execution and the chain rule of differential calculus to compute derivatives. However, this technique is not applicable to our framework due to computational complexity in backpropagation. Our PIGNN proceeds in an iterative manner, and the gradient of the solution function at each node depends on historical results up to the initial time, which leads to memory resource pressure and computational complexity. On a regular grid, the gradient of solution $\bm{{\rm u}}$ at node $\bm{{\rm x}}_i$ is a vector whose components are the derivatives of $\bm{{\rm u}}$ at $\bm{{\rm x}}_i$ along the axes. Each partial derivative can be computed by stencils of backward/forward/central finite difference. However, on an irregular mesh, it is infeasible to directly calculate the derivative along the axis. On the basis of least square method, we propose a simple gradient estimation method. Using the first-order Taylor expansion, we can construct a linear system. An estimation of the gradient is then obtained by learning the best fitting solution (least squares solution) of the linear system.

Let ${\bm{{\rm x}}}_i=({\rm x}_i^1,\cdots,{\rm x}_i^d)^{\rm T}\in {\mathbb R}^d,
\,\bm{{\rm u}}({\bm{{\rm x}}})= \bm{{\rm u}}({\bm{{\rm x}}},t)$ for any fixed $t$, $\,\mathcal{N}({\bm{{\rm x}}_i})=\{{\bm{{\rm x}}}_{j1}, \cdots, {\bm{{\rm x}}}_{jp}\}$ be the node set of one-ring neighbours of node ${\bm{{\rm x}}}_i$, $p=|\mathcal{N}({\bm{{\rm x}}_i})|$ be the number of neighbours of node ${\bm{{\rm x}}}_i$.

According to the first-order Taylor expansion of solutions at node $\bm{{\rm x}}_i$, we have
\begin{eqnarray}\label{Eq3.2}
\bm{{\rm u}}({\bm{{\rm x}}}_j)- \bm{{\rm u}}({\bm{{\rm x}}}_i) =
({\bm{{\rm x}}}_j-{\bm{{\rm x}}}_i)^{\rm T}\nabla \bm{{\rm u}}({\bm{{\rm x}}_i}),
\,\, \forall j\in\mathcal{N}({\bm{{\rm x}}_i}).
\end{eqnarray}

Without loss of generality, we assume that the number of neighbours is larger than the dimension of the above linear system, the system is overdetermined and only admits a least squares solution. Utilizing the least squares technique, we obtain an estimation of the gradient by solving the linear system.

Let $\bm{A}_{i}$ be the $p\times d$ matrix obtained by collecting all edges $\bm{{\rm x}}_j-\bm{{\rm x}}_i$, $\bm {U}_i$ be the column vector consisting of differences of node solutions connecting the corresponding edges $\bm{{\rm u}}({\bm{{\rm x}}}_j)- \bm{{\rm u}}(\bm{{\rm x}}_i)$, concretely,
\begin{eqnarray*}
\bm{A_{i}} &=&
\left(\begin{array}{cc}
    \big(\bm{{\rm x}}_{j1}-\bm{{\rm x}}_i\big)^{\rm T} \\
    \vdots                                     \\
    \big(\bm{{\rm x}}_{jp}-\bm{{\rm x}}_i\big)^{\rm T}\\
\end{array}\right)
=\,
\left(\begin{array}{ccc}
	{\rm x}_{j1}^{1} - {\rm x}_{i}^{1} & \cdots &
	{\rm x}_{j1}^{d} - {\rm x}_{i}^{d} \\
	\vdots & \ddots & \vdots \\
	{\rm x}_{jp}^{1} - {\rm x}_{i}^{1} & \cdots &
	{\rm x}_{jp}^{d} - {\rm x}_{i}^{d} \\
\end{array}\right)_{p\times d} ,
\end{eqnarray*}
\begin{eqnarray*}
\bm{U_i} &=&
\left(\begin{array}{c}
    \bm{{\rm u}}(\bm{{\rm x}}_{j1} )-\bm{{\rm u}}({\bm{{\rm x}}_i} )  \\
    \vdots                   \\
    \bm{{\rm u}}(\bm{{\rm x}}_{jp} )-\bm{{\rm u}}({\bm{{\rm x}}_i} )  \\
\end{array}\right)_{p\times 1}.
\end{eqnarray*}

Then the linear system Eq. \eqref{Eq3.2} can be written as $\bm {U}_i=\bm{A}_{i}\nabla \bm{{\rm u}}({\bm{{\rm x}}_i})$, and its least square solutions, namely, the estimation of the gradient at node $\bm{{\rm x}}_i$ is
\begin{equation}\label{Eq3.4}
\nabla \bm{{\rm \hat{u}}}({\bm{{\rm x}}_i})  =
 \big({\bm A}_i^{\rm T}{\bm A}_i\big)^{-1}{\bm A}_i^{{\rm T}} {\bm U}_i.
\end{equation}

\subsubsection{Laplace operator}
\noindent{\textbf{\\}}

Another commonly used differential operator is the Laplace operator. It is worth noting that the following method we proposed is general, which can be used to calculate other high-order differential operators similarly. Commonly used methods to compute the Laplace operator on graphs (unstructured meshes) include Taubin's discretization \cite{Taubin1995, Taubin2000}, Fujiwara's discretization \cite{Fujiwara1995}, Desbrun et al.'s discretization \cite{Desbrun1999}, Mayer's discretization \cite{Mayer2001}, and Meyer et al.'s discretization \cite{MeyerDesbrun2003}. All these methods adopt a general form which can be expressed in our notations as:
\begin{equation}\label{Eq3.6}
	\Delta \bm{{\rm u}}(\bm{{\rm x}}_i,t)=\sum_{j\in \mathcal{N}({\bm{{\rm x}}_i})}
	{\omega_{ij}}\big[\bm{{\rm u}}(\bm{{\rm x}}_j,t)-\bm{{\rm u}}(\bm{{\rm x}}_i,t)\big],
\end{equation}
where $\Delta={\rm div}(\nabla) =\nabla\cdot \nabla $ denotes the Laplace operator, $\bm{\omega}_{i}= (\omega_{i1},\cdots,\omega_{ip})^{\rm T}$ is the coefficient/weight vector to be determined.

\smallskip
\noindent{\textbf{Remark 1.}} We can interpret the Eq. \eqref{Eq3.6} in the following way. Recalling that the Gauss Divergence Theorem says, the surface integral of a vector field $\bm{F}$ over a closed surface $S$ is equal to the volume integral of the divergence of the vector field $\bm{F}$ over the volume $V$ enclosed by the closed surface, i.e., $\iint_{S} \bm{F}\cdot \vec{\bm{n}}\,{\rm d}S =\iiint_{V} (\nabla \cdot\bm{F}){\rm d}V $. By specializing the theorem to the plane, then the following special equation holds,
\begin{equation}\label{Eq3.7}
\int_{\partial S} \bm{F}\cdot \vec{\bm{n}}\,{\rm d}s =\iint_{S} (\nabla \cdot\bm{F}){\rm d}S.
\end{equation}

Letting $\bm{F}= \nabla\bm{{\rm u}}$ and paying attention to the Eq. \eqref{Eq3.2}, we see that, taking the small area around node $\bm{{\rm x}}_i$ as the domain, according to Eq. \eqref{Eq3.7}, the Laplace value of node $\bm{{\rm x}}_i$ can alternatively be approximated by the sum of directional derivatives along the domain boundary, which is approximated by Eq. \eqref{Eq3.6}.\qquad\qquad\quad$\Box$

Noting from the coefficient expressions of classical discretizations, such as combinatorial weights, mean-value coordinates weights, and cotan weights, we see that coefficients only related to the spatial locations of adjacent nodes, and have nothing to do with values of solutions. That is to say, for the PDE system, when the discretization scheme on the computational domain is fixed, the generated mesh is determined, and the coefficients $\bm{\omega}$ is also determined and remains unchanged. On this basis, we estimate the coefficients by selecting a set of test functions and applying the ordinary least squares technique.

In order to calculate the second order derivatives at node $\bm{{\rm x}}_i$, we choose a set of basis functions whose degree don't exceed $m\,(m\geq2)$ as test functions $\tilde{\bm{{\rm u}}}^{k},\,k=1,\cdots,K$, where $K$ is the number of test functions. These basis functions are terms (excluding the constant term) of a $d$-element polynomial of $m$-th degree, i.e., $\tilde{\bm{{\rm u}}}^{k}(x,y,t) \in \Psi_d=\{x_1,\cdots,x_d,x_1^2,x_1x_2,\cdots,x_d^2,\cdots,x_1^m,\cdots,x_d^m\}$.

Based on the selected test functions, by finding the least squares solution of Eq. \eqref{Eq3.6}, we obtain the following estimation of the coefficient vector,
\begin{equation}\label{Eq3.9}
\bm{\hat{\omega}}_i=\big(\bm{C}_{i}^{{\rm T}}\bm{C}_{i}\big)^{-1}\bm{C}_{i}^{{\rm T}}\bm{B}_i,
\end{equation}
where,
\begin{eqnarray*}
\qquad
\bm{\hat{\omega}}_i &=& [\hat{\omega}_{i1},\cdots,\hat{\omega}_{ip}]^T,\qquad\qquad\quad\,\,
\tilde{\bm{{\rm u}}}_{ji}^{k} =
\tilde{\bm{{\rm u}}}^{k}(\bm{{\rm x}}_j,t)-\tilde{\bm{{\rm u}}}^{k}(\bm{{\rm x}}_i,t),\\
\bm{C}_i &=&
\left(\begin{array}{ccc}
\tilde{\bm{{\rm u}}}_{1i}^{1} & \cdots & \tilde{\bm{{\rm u}}}_{pi}^{1}   \\
\vdots             & \ddots  & \vdots                \\
\tilde{\bm{{\rm u}}}_{1i}^{K} & \cdots & \tilde{\bm{{\rm u}}}_{pi}^{K}   \\
\end{array}\right)_{K\times p},\quad
\bm{B}_i =
\left(\begin{array}{c}
\Delta\tilde{\bm{{\rm u}}}^{1}(\bm{{\rm x}}_i,t)  \\
\vdots                   \\
\Delta\tilde{\bm{{\rm u}}}^{K}(\bm{{\rm x}}_i,t)  \\
\end{array}\right)_{K\times 1}. \qquad\qquad\qquad
\end{eqnarray*}

Substituting \eqref{Eq3.9} into \eqref{Eq3.6} gives us the estimation of the Laplace operator at node $\bm{{\rm x}}_i$.

\smallskip
\noindent{\textbf{Remark 2.}}

(1) To eliminate the influence of $\bm{{\rm x}}_i$ on ${\omega_{ij}}$ and improve computational stability, we translate the coordinate system and compute $\bm{{\rm u}}(\bm{{\rm x}}_j-\bm{{\rm x}}_i,t) $ instead of $\bm{{\rm u}}(\bm{{\rm x}}_j,t)$, then Eq. \eqref{Eq3.6} becomes
\begin{equation}\label{Eq3.9.1}
\Delta \bm{{\rm u}}(0,t)=\sum_{j\in \mathcal{N}({\bm{{\rm x}}_i})}{\omega_{ij}}
\big[\bm{{\rm u}}(\bm{{\rm x}}_j-\bm{{\rm x}}_i,t)-\bm{{\rm u}}(0,t)\big].
\end{equation}

(2) In general, in a $d$-dimensional space (i.e, $d$ variables), we can calculate that the number of non-constant terms for a test function with no more than $m$-th degree is $K=C_{d+m}^{m}-1$, then to make Eq.\eqref{Eq3.6} have a solution, the degree of the test function needs to satisfy $C_{d+m}^{m}\geq p+1$, where $p$ is the number of neighbours of the current node. \qquad\qquad\qquad\qquad\qquad\qquad\qquad\qquad\qquad\qquad\qquad\qquad\qquad\qquad\qquad\qquad\,\,$\Box$

\begin{definition}\label{def1}[Order of approximation]
Let $\bm{{\rm x}}_0\in {\mathbb R}^p$, $\bm{{\rm x}}_1,\cdots,\bm{{\rm x}}_p\in\mathcal{N}(\bm{{\rm x}}_0, R)=\big\{\bm{{\rm x}} | \|\bm{{\rm x}}-\bm{{\rm x}}_0\|\leq R\big\})$.
Define linear operator $\bm{{\rm \omega}}$ as follows,
\begin{equation}\label{Eq3.9.2}
\bm{{\rm \omega}}\bm{{\rm u}}(\bm{{\rm x}}_0) = \sum\limits_{k=1}^p {\omega_{k}}
\big[\bm{{\rm u}}(\bm{{\rm x}}_k)-\bm{{\rm u}}(\bm{{\rm x}}_0)\big],
\,\,{\rm for}\,\, \bm{{\rm u}}\in C^2\big(\mathcal{N}(\bm{{\rm x}}_0, R)\big).
\end{equation}
If for all $\bm{{\rm u}}\in C^2\big(\mathcal{N}(\bm{{\rm x}}_0, R)\big)$, we have
\begin{equation}\label{Eq3.9.3}
\Delta \bm{{\rm u}}(\bm{{\rm x}}_0)-\sum\limits_{k=1}^p {\omega_{k}}
\big[\bm{{\rm u}}\big(\bm{{\rm x}}_0+r(\bm{{\rm x}}_k-\bm{{\rm x}}_0)\big)-\bm{{\rm u}}(\bm{{\rm x}}_0)\big]
=o(r^s), \,\,{\rm as}\,\, r\rightarrow0.
\end{equation}

Then we say $\bm{{\rm \omega}}$ is an approximation of Laplacian $\Delta$ with order of magnitude $s$ at $\bm{{\rm x}}_0$.
\end{definition}

\begin{lemma}\label{lemma3}
Let $\,\bm{{\rm x}}_0=\bm{0}$, $\bm{{\rm x}}_1,\cdots,\bm{{\rm x}}_p\in\mathcal{N}(\bm{0}), \,d=2,\, \Psi=\Psi_2=\{x, y, x^2, xy,y^2,\cdots,\\ x^m,\cdots,y^m\}$, $\varphi_1=x,\, \varphi_2=y,\cdots$, $|\Psi|=n$, $\bm{\omega}_r$ be the least square solution of the following linear system
\begin{equation}\label{Eq3.9.4}
\bm{{\rm Y}}=\Phi(r)\bm{\omega},
\end{equation}
where
\begin{eqnarray*}
&&
\bm{\omega}=(\omega_1,\cdots,\omega_p)^{\rm T}, \qquad\quad
\bm{{\rm Y}} = \big(\Delta\varphi_1(\bm{0}),\cdots, \Delta\varphi_n(\bm{0})\big)^{\rm T},
\qquad\qquad\qquad \\
&&
\bm{\Phi}(r)=
\left(\begin{array}{ccc}
\varphi_1(r\bm{{\rm x}}_{1}) & \cdots  & \varphi_1(r\bm{{\rm x}}_{p})   \\
\vdots             & \ddots  & \vdots                \\
\varphi_n(r\bm{{\rm x}}_{1}) & \cdots  & \varphi_n(r\bm{{\rm x}}_{p})  \\
\end{array}\right)_{n\times p}, \quad
\bm{\Phi} = \bm{\Phi}(1).
\end{eqnarray*}
Let $\bm{\eta}_r=r^2\bm{\omega}_r$, then

1) $\bm{\eta}_r$ is bounded.

2) Let $\{\bm{\eta}_n\}_{n=1}^{\infty}$ be any convergent subsequence of $\bm{\eta}_r$, if $\bm{\eta} = \lim\limits_{n\rightarrow\infty}\bm{\eta}_n$, then $\bm{\eta}$ is an approximation of Laplacian $\Delta$ with the largest order.

\end{lemma}
\begin{proof}
See Appendix A.
\end{proof}

\subsection{Interpolation method for inverse problem}
\noindent{\textbf{\\}}

In this section, we describe the interpolation method adopted in the inverse problem. Recall that in the inverse problem, some PDE parameters are unknown, but a few observations are available. For a given time point $t_d$ and spatial location $\bm{{\rm x}}_d$, the observed value $\bm{{\rm y}}(\bm{{\rm x}}_d,t_d)$ is known. By minimizing the difference between observations and PIGNN-based predictions at these given spatiotemporal points, we can estimate the unknown parameters. Noting that $t_d$ may not equal $t_p$, to calculate the data loss, we need to know the predicted values at any time step $t$ and spatial location $\bm{{\rm x}}$, i.e., $\bm{{\rm u}}_{pred}(\bm{{\rm x}}, t; \bm{\lambda, \theta})$. For simplicity, in the following, we denote $\bm{{\rm u}}_{pred}(\bm{{\rm x}}, t)=\bm{{\rm u}}_{pred}(\bm{{\rm x}}, t; \bm{\lambda, \theta})$.

Assume that $t_d\in[t_p,t_{p+1}),\,\,\mathcal{N}(\bm{{\rm x}}_d)=\{\bm{{\rm x}}_{1},\cdots,\bm{{\rm x}}_{l}\}$ be $l$ nearest nodes around $\bm{{\rm x}}_d$. According to the Taylor expansion, we have
\begin{eqnarray}
\bm{{\rm u}}(\bm{{\rm x}}_m, t_n) &\approx&
\bm{{\rm u}}(\bm{{\rm x}}_d,t_d)
+\frac{\partial \bm{{\rm u}}(\bm{{\rm x}}_d,t_d)}{\partial t}(t_n-t_d)
+ (\bm{{\rm x}}_m-\bm{{\rm x}}_d)^{\rm T}\nabla \bm{{\rm u}}(\bm{{\rm x}}_d,t_d),  \label{Eq3.10}\\
&&
\,\, \bm{{\rm x}}_m\in\mathcal{N}(\bm{{\rm x}}_d),\,\, n=p,p+1.  \nonumber
\end{eqnarray}

The estimation of the desired prediction and its derivatives can be obtained with the least square technique, which is given by
\begin{equation}\label{Eq3.11}
\bm{\hat{H}}=\big(\bm{Q}^{{\rm T}}\bm{Q}\big)^{-1}\bm{Q}^{{\rm T}}\bm{Z},
\end{equation}
where,
\begin{eqnarray*}
\bm{\hat{H}} =
\left(\begin{array}{c}
\bm{{\rm u}}_{pred}(\bm{{\rm x}}_d,t_d)  \\
\frac{\partial \bm{{\rm u}}_{pred}(\bm{{\rm x}}_d,t_d)}{\partial t}                  \\
\nabla \bm{{\rm u}}_{pred}(\bm{{\rm x}}_d,t_d)  \\
\end{array}\right), \qquad
\bm{Z} =
\left(\begin{array}{c}
\bm{{\rm u}}(\bm{{\rm x}}_1,t_p) \\
\vdots                   \\
\bm{{\rm u}}(\bm{{\rm x}}_l,t_p) \\
\bm{{\rm u}}(\bm{{\rm x}}_1,t_{p+1}) \\
\vdots                   \\
\bm{{\rm u}}(\bm{{\rm x}}_l,t_{p+1}) \\
\end{array}\right)_{2l\times 1},\\
\bm{Q} =
\left(\begin{array}{ccc}
1      & t_p-t       & (\bm{{\rm x}}_1-\bm{{\rm x}})^{\rm T}  \\
\vdots               & \vdots  & \vdots                \\
1      & t_p-t       & (\bm{{\rm x}}_l-\bm{{\rm x}})^{\rm T}   \\
1      & t_{p+1}-t   & (\bm{{\rm x}}_1-\bm{{\rm x}})^{\rm T}   \\
\vdots               & \vdots  & \vdots                \\
1      & t_{p+1}-t   & (\bm{{\rm x}}_l-\bm{{\rm x}})^{\rm T}  \\
\end{array}\right)_{2l\times (d+2)}.\qquad\qquad\qquad\qquad
\end{eqnarray*}

From Eq. \eqref{Eq3.11}, we can obtain the predicted value $\bm{{\rm u}}_{pred}(\bm{{\rm x}},t)$ at any spatial-temporal point $(\bm{{\rm x}},t)$, therefore, for the observation point $(\bm{{\rm x}}_d, t_d)$, its prediction $\bm{{\rm u}}_{pred}(\bm{{\rm x}}_d, t_d)$ is known, and data loss can be calculated.

\section{Experiments}\label{sec4}
\noindent{\textbf{\\}}

In this section, we conduct a series of experiments on a group of canonical PDEs in both forward and inverse settings to demonstrate the effectiveness of our proposed PIGNN. We aim to answer the following questions.

\begin{enumerate}[Q1.]
\item
How accurate is the proposed PIGNN in solving PDEs compared with the baseline model?\smallskip
\item
How does PIGNN perform in time extrapolation? (i.e., model settings are consistent with that of training time steps)?\smallskip
\item
Does PIGNN generalize well to other conditions not seen during training, such as other computation domains, spatial resolutions and ICs?
\end{enumerate}

To answer these questions, we choose three nonlinear PDE systems ranging from simple to complex: heat equation, Burgers equation and FitzHugh-Nagumo (FN) equation in 2-dimensional (2D) domains with known/unknown parameters, and take the vanilla PINN as the baseline model. Heat equation is probably one of the simplest PDEs to describe the behavior of diffusive systems. Burgers equation \cite{Burgers1948} is an important nonlinear parabolic PDE widely used to model fluid dynamics, which directly incorporates the interaction between non-linear convection processes and diffusive viscous processes. The coupled FN equation \cite{FitzHugh1955} is an example of a nonlinear reaction-diffusion system that is introduced for the nerve impulse propagation in nerve axons to describe activator-inhibitor neuron activities excited by external stimulus.

The effectiveness of PIGNN for other equations such as Euler, Navier-Stokes and Maxwell is under investigation and related results will be presented in our future studies. It is rather remarkable that though PINN learns the global solution directly from training points, and PIGNN is a time-step iterative method. PINN is chosen as the baseline model for comparison because we believe that both methods can learn propagation trends of physical systems over the training time steps.

We consider the initial-boundary value problem with Dirichelet BCs in the computational domains for all examples. To verify the accuracy of our model, we choose specific ICs and BCs for heat and Burgers equations such that they have analytical solutions. The concrete settings for PIGNN and baseline model are illustrated in Appendix B. And concrete expressions of PDE examples are shown in Appendix C.

\subsection{Evaluation metrics}
\noindent{\textbf{\\}}

To evaluate the accuracy of predicted solutions, at time step $t$, we use the absolute error between the predicted solution $\bm{{\rm u}}_{pred}(\bm{{\rm x}},t;\bm{\theta})$ and the real solution $\bm{{\rm u}}^{*}(\bm{{\rm x}},t)$, given by,
\begin{equation}\label{Eq4.1}
{\rm Error} = |\bm{{\rm u}}_{pred}(\bm{{\rm x}},t;\bm{\theta})-\bm{{\rm u}}^{*}(\bm{{\rm x}},t)|.
\end{equation}

While for the entire testing period, we use the accumulative root-mean-square-error (aRMSE) which is given by
\begin{equation}\label{Eq4.2}
{\rm aRMSE} = \sqrt{\frac{1}{N N_t }
\sum_{k=1}^{N_t}{\sum_{i=1}^{N}{ \|\bm{{\rm u}}_{pred}(\bm{{\rm x}}_i,t_k;\bm{\theta})-\bm{{\rm u}}^{*}(\bm{{\rm x}}_i,t_k)\|_2^2 }}}  ,
\end{equation}
where $N_t$ is the total number of time steps within $[0,t]$.

\subsection{Results}
\noindent{\textbf{\\}}

In this subsection, we present results for three PDEs in both forward and inverse settings. Specifically, we start by showing results of the forward problem with long time steps to illustrate the accuracy and time extrapolability of our PIGNN. Then, a series of results with varying PDE parameters, computational domains and spatial resolutions are presented to demonstrate the generalizability of PIGNN. Moreover, we present parallel results to demonstrate that PIGNN can scale well to distributed multi-GPU systems to handle large-scale problems. In addition, a simple result verifies the ability of PIGNN to handle inverse problems. Finally, we conduct an ablation study to obtain a deeper understanding of our PIGNN. We display two experiments to elucidate the effects of the number of GN blocks and proper node features on model performance.

Our main finding is that PIGNN outperforms the PINN baseline for almost all cases, and in particular it generalizes better to more complex situations. The main advantage of our method is that models trained in small domains with simple settings possess fantastic fitting ability, and can be directly applied to more complex situations with large domains. This is extremely important in practice, as there is often a demand for rapid assessment or response. When the computational domain, ICs or BCs change, traditional numerical methods need to be recalculated, which is very time-consuming. DL-based methods with poor generalization ability may produce incorrect predictions. While our method can quickly give high confidence estimations.

\subsubsection{Extrapolability of the forward problem}
\noindent{\textbf{\\}}

We train an independent PIGNN/PINN model for each PDE example to obtain approximate solutions under the corresponding training settings. Specifically, for heat equation, we train the PIGNN model for 50 time steps with a time interval $\Delta t = 10^{-4}$, while for Burgers and FN equations, we only let the PIGNN model learn 10 time steps with $\Delta t = 10^{-3}$. We take the model with the least loss as the test model to evaluate the model's performance, hoping that it can infer the subsequent solutions. During testing, we predict hundreds or thousands of time steps for each example to demonstrate the time extrapolability of our PIGNN. For fair comparison, in PINN training the computation area and training time steps are the same as those in our approach. We randomly select a certain number of collocation points from both the computation domain and temporal scope $[0, {\rm{time\,\, steps}}\times\Delta t]$. Since PINN is meshless, when testing, the inputs to the PINN model are mesh points' coordinates of the corresponding PIGNN model and time step. That is, PIGNN's mesh points are taken as the domain sampling points.

Figures \ref{fig3}-\ref{fig5} describe the forward solution snapshots for each example predicted by PIGNN and PINN, compared against the analytical solutions. It is worth noting that there is no analytical solution to the FN equation, and we take results of the numerical commercial software as the reference. At the same time, to show the accuracy for each method, we provide snapshots of the absolute error at the corresponding time steps. Besides, we present the accumulative error curves with respect to time in Fig. \ref{fig6} to further depict the accuracy of different methods.

In Fig \ref{fig3}, we select three representative snapshots of heat equation at time steps 50, 1000 and 2500, respectively. It can be seen that our PIGNN shows excellent agreement with the analytical solution in both training and extrapolation phases, while PINN fails to match the analytical result at longer time steps. It is clear from the last two columns of the figure that the absolute error of PINN is much larger than that of PIGNN. This phenomenon is further verified by the left subplot of Fig. \ref{fig6}, where the aRMSE curve of PIGNN is consistently lower than that of PINN.

Figure \ref{fig4} exhibits similar performance for the Burgers equation, where we present predictions at time steps 5, 500 and 1500. As can be seen from the right subplot of Fig. \ref{fig6}, the aRMSE of PIGNN for the Burgers equation starts out relatively larger than that of PINN, then the curve flattens out and drops below the PINN curve at half steps. Since we fail to train a PINN model for the FN equation, we only present the extrapolation results of PIGNN at time steps 5 and 500 in Fig. \ref{fig5}, compared with the reference solution. Similar to the above two cases, the minor difference in training and extrapolation phases demonstrate the great learning potential of our PIGNN.

\begin{figure}[!htbp]
\begin{center}
\scalebox{0.5}[0.5]{\includegraphics{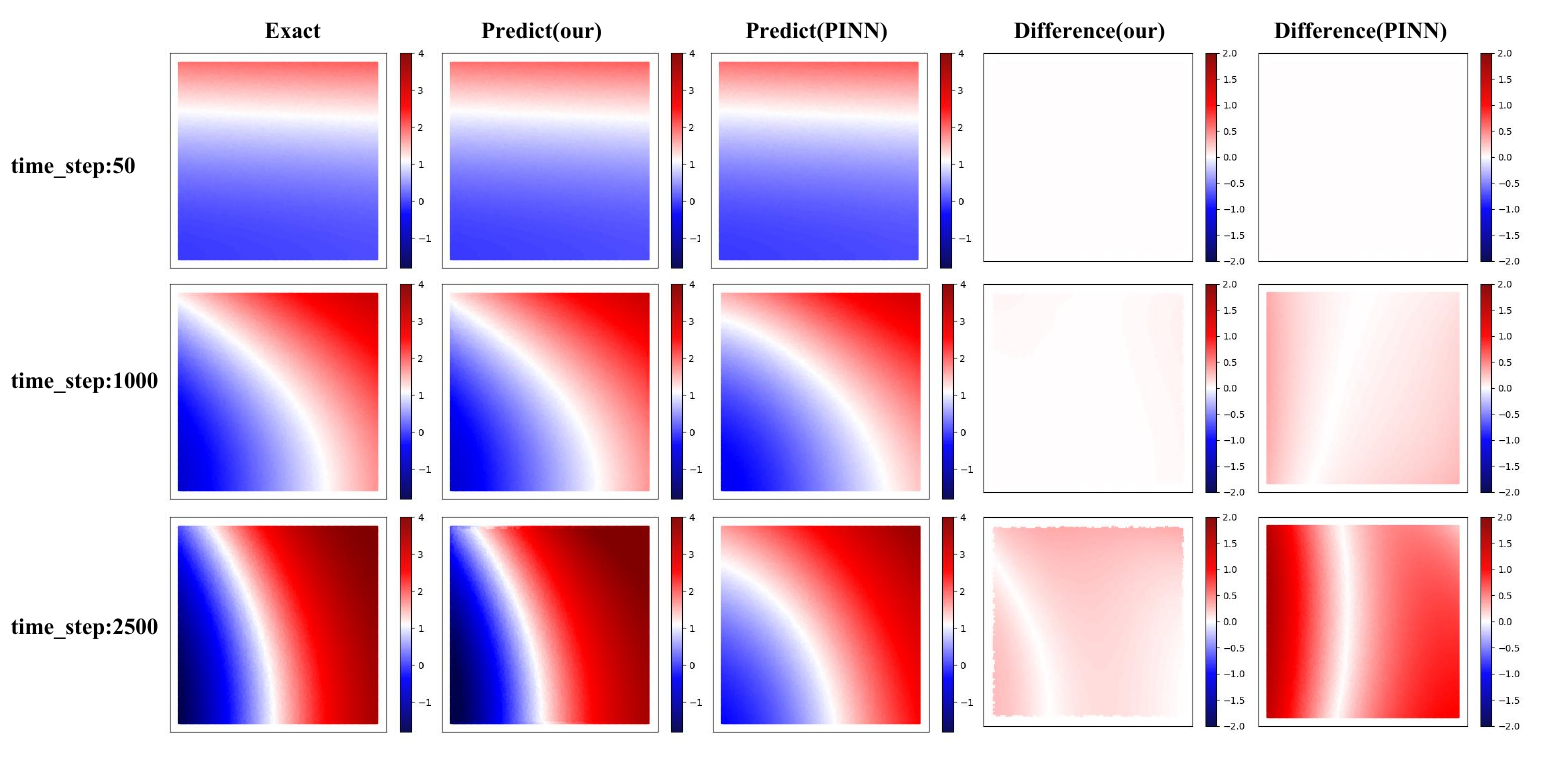}}
\caption{\footnotesize{Comparison of time extrapolability between PIGNN and PINN for heat equation.
		Computational domain: $[0,1]\times[0,1]$, mesh density=100.}}
\label{fig3}
\end{center}
.
\end{figure}

\begin{figure}[!htbp]
\begin{center}
\scalebox{0.55}[0.55]{\includegraphics{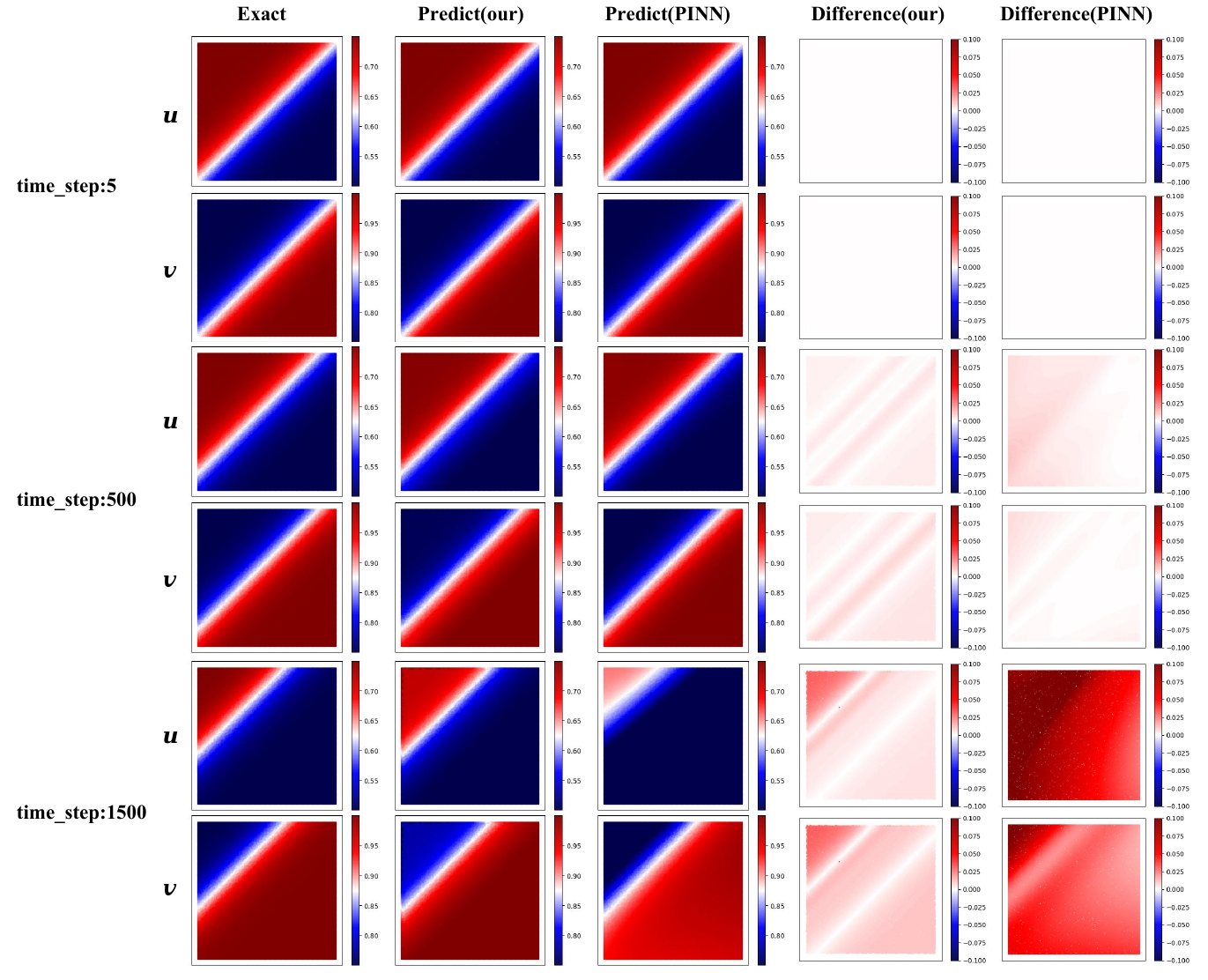}}
\caption{\footnotesize{Comparison of time extrapolability between PIGNN and PINN for Burgers equation.
Computational domain: $[0,1]\times[0,1]$, mesh density=100.}}
\label{fig4}
\end{center}
\end{figure}

\begin{figure}[!htbp]
\begin{center}
\scalebox{0.55}[0.55]{\includegraphics{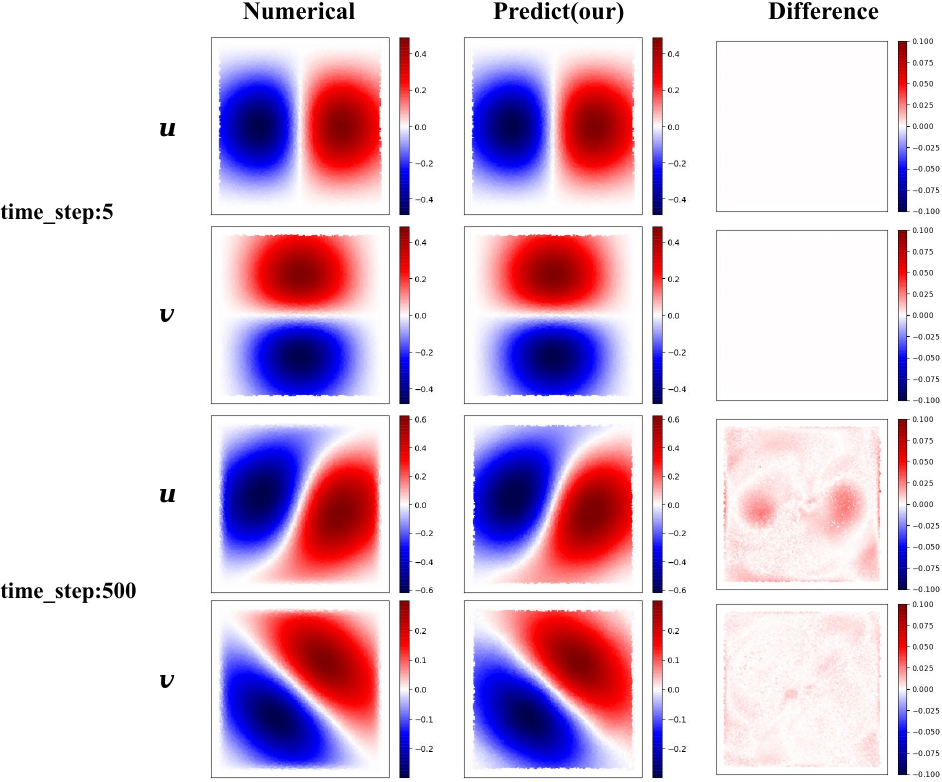}}
\caption{\footnotesize{Time extrapolability of PIGNN compared with numerical results for FN equation.
Computational domain: $[0,1]\times[0,1]$, mesh density=100.}}
\label{fig5}
\end{center}
\end{figure}

\begin{figure}[!htbp]
\begin{minipage}[t]{0.485\linewidth}
\centerline{\includegraphics[width=1.15\textwidth]{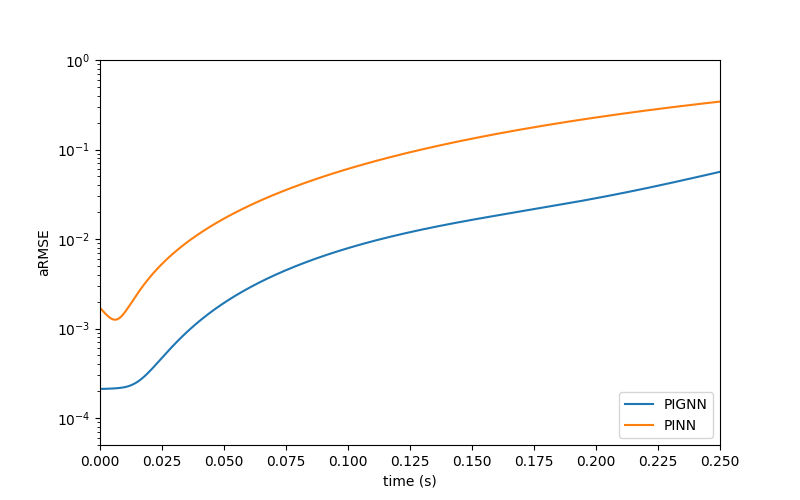}}
\end{minipage}
\hfill
\begin{minipage}[t]{0.485\linewidth}
\centerline{\includegraphics[width=1.15\textwidth]{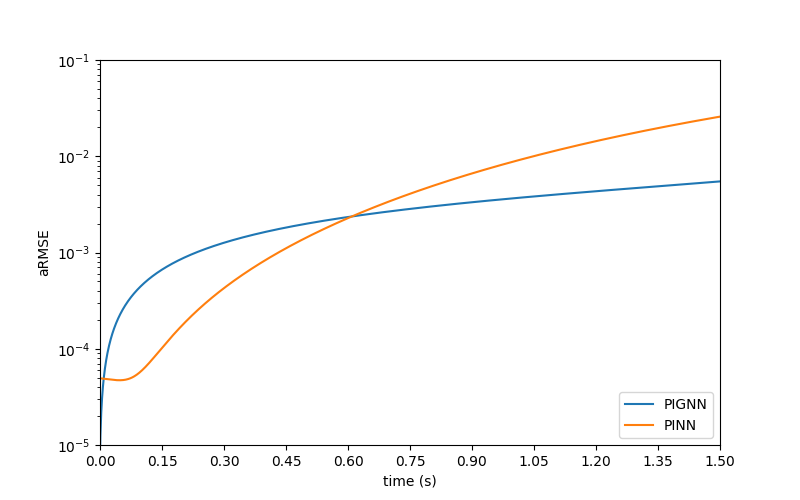}}
\end{minipage}
\caption{\footnotesize{Comparison of aRMSE curves between PIGNN and PINN
 for heat (left) and Burgers (right) equations. Computational domain: $[0,1]\times[0,1]$, mesh density=100.}}
\label{fig6}
\end{figure}

\subsubsection{Generalization on the forward problem}
\noindent{\textbf{\\}}

As aforementioned, fast estimations or responses are often required in practice. When computational domains, ICs or BCs change, numerical methods need to be recalculated, which is very time-consuming. DL-based methods with poor generalizability may produce incorrect estimations. Therefore, in this subsection, by carrying out a series of experiments, we demonstrate that our PIGNN generalizes well beyond training scenarios for varying PDE parameters, computational domains and mesh sizes (or spatial resolutions).

\begin{enumerate}
\smallskip
\item [\quad\textbf{(1)}] \textbf{Changing PDE parameters}
\smallskip
\end{enumerate}

Firstly, we change the PDE parameter in heat equation from $b=2$ to $b=8$. From Eq. \eqref{Eq4.3} and the definition of BC, we known that different parameters actually mean different PDE source terms $f(t,x,y)$ and BCs. Fig. \ref{fig7} shows a visual comparison between PIGNN, PINN and analytical solutions. It can be clearly seen that PINN is unable to produce correct results, while our PIGNN performs well even for long time steps.

\begin{figure}[!htbp]
\begin{center}
\scalebox{0.5}[0.5]{\includegraphics{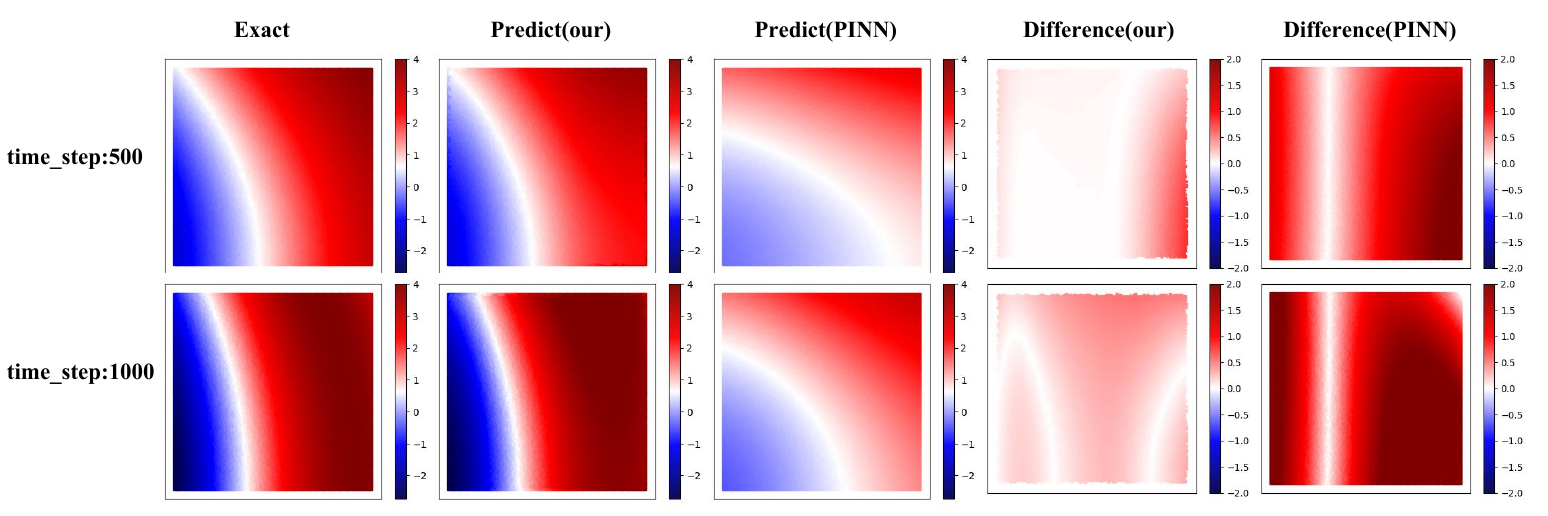}}
\caption{\footnotesize{Generalization performance of the heat equation with parameter $b$ changed from 2 to 8.
Computational domain: $[0,1]\times[0,1]$, mesh density=100.}}
\label{fig7}
\end{center}
\end{figure}

\begin{enumerate}
\smallskip
\item [\quad\textbf{(2)}] \textbf{Changing spatial resolutions}
\smallskip
\end{enumerate}

Secondly, we maintain the same size of computational domain and increase/decrease the mesh size, namely, the spatial resolution is finer/coarser than the default settings. Since PINN is meshless, changing resolutions of the PINN model implies adopting another sampling strategy. Figures \ref{fig8}-\ref{fig9} show snapshots of solutions to heat and Burgers equations predicted by PINN and PIGNN under finer and coarser resolutions at time steps 500 and 1500, respectively. For the finer resolution case, we increase the mesh density from 100 to 200, while in the coarser case, we decrease the mesh density from 100 to 50. As can be seen from the last two columns of each figure, the absolute error of PINN is larger than that of PIGNN for both equations. This phenomenon illustrates that our PIGNN is much more robust to changes in spatial resolution.

\begin{figure}[!htbp]
\centering
\subfigure[heat]{\includegraphics[width=0.95\linewidth]{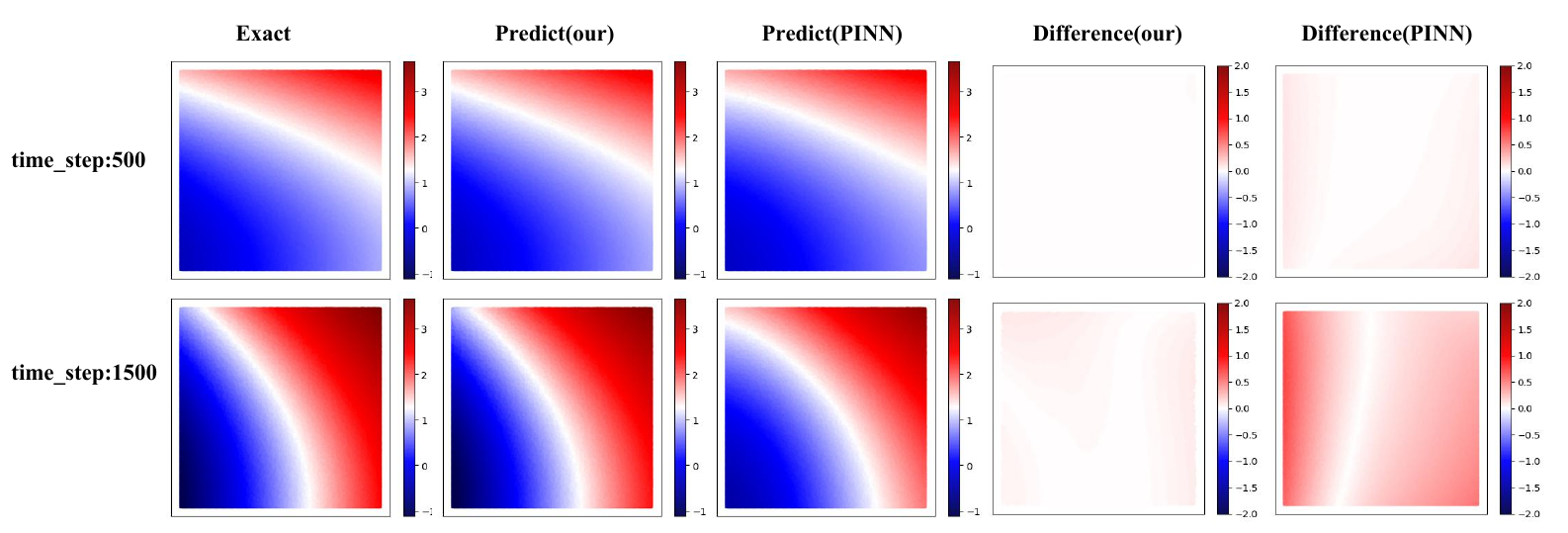}}
\subfigure[Burgers]{\includegraphics[width=0.95\linewidth]{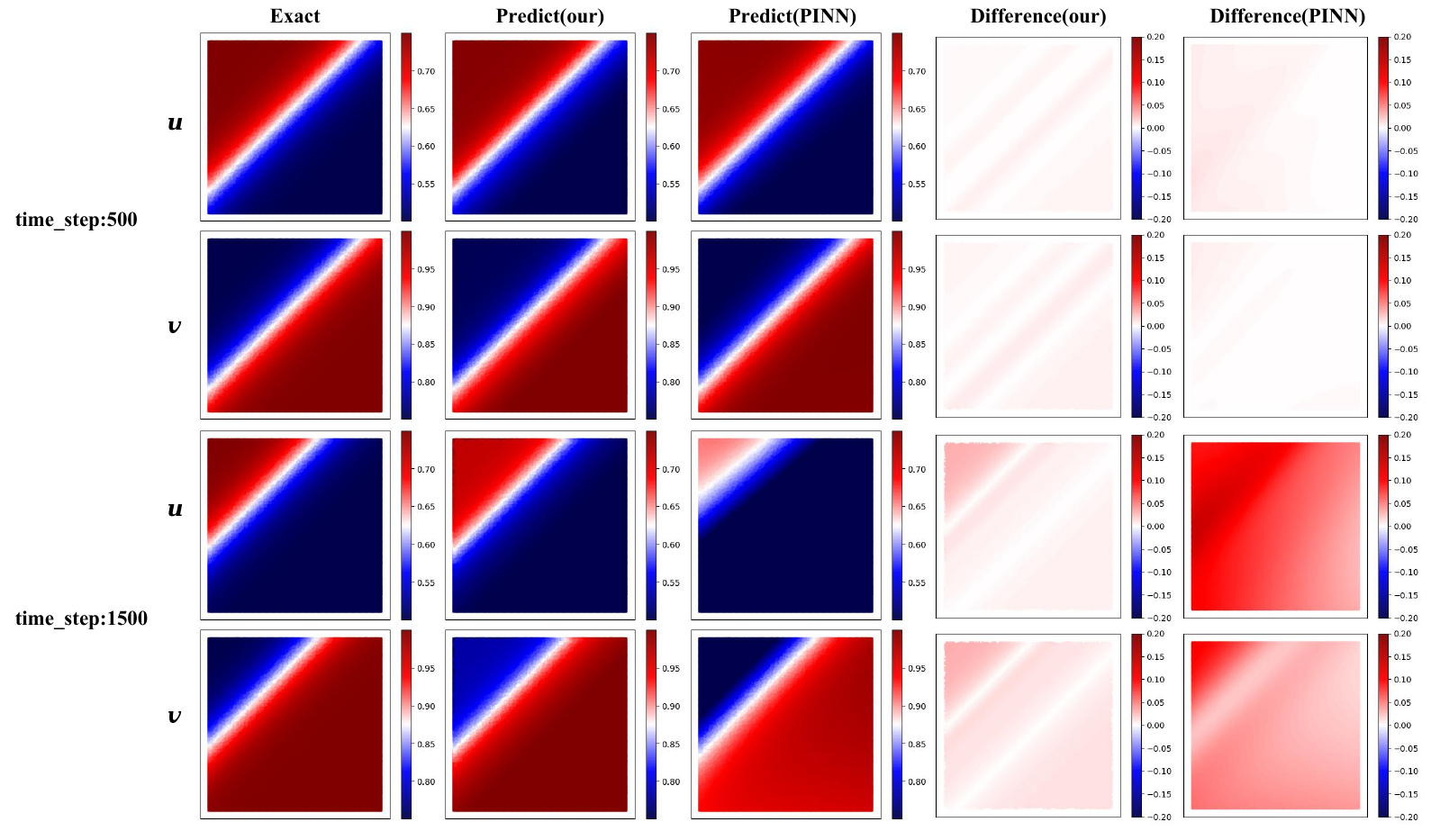}}
\caption{\footnotesize{Generalization performances of heat (top) and Burgers (bottom) equations under the default computational domain with finer spatial resolution. Computational domain: $[0,1]\times[0,1]$, mesh density=200 (100 by default).}}
\label{fig8}
\end{figure}

\begin{figure}[!htbp]
\centering
\subfigure[heat]{\includegraphics[width=0.95\linewidth]{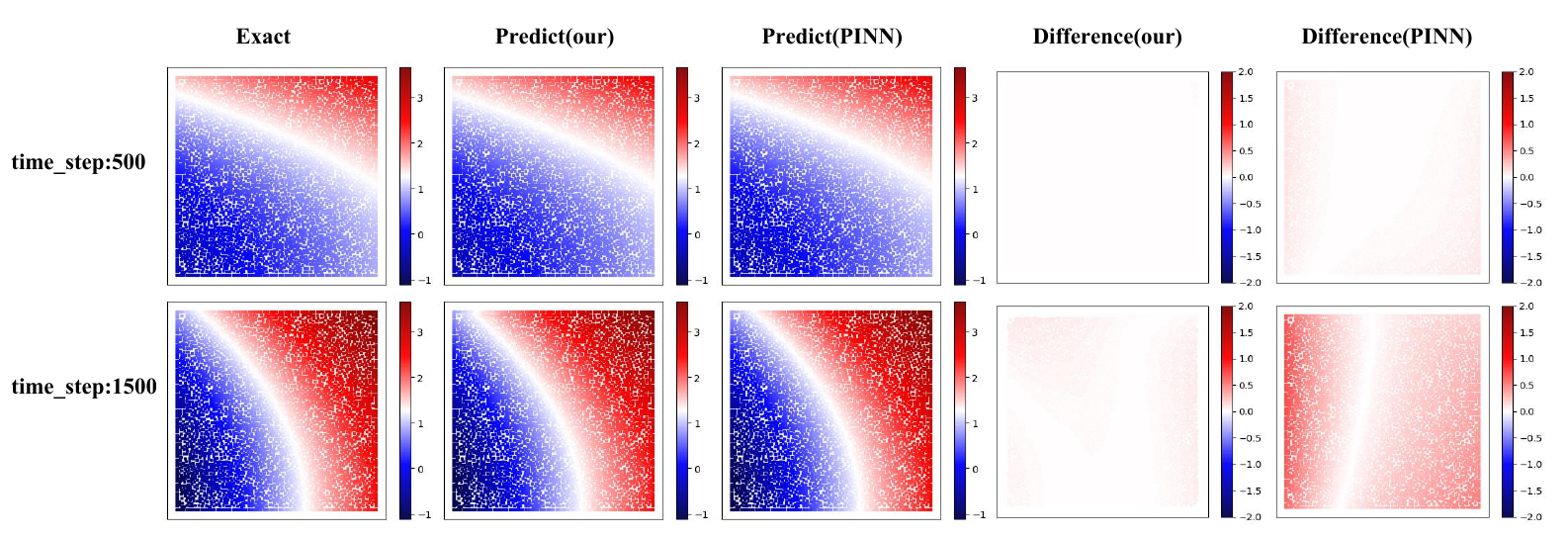}}
\subfigure[Burgers]{\includegraphics[width=0.95\linewidth]{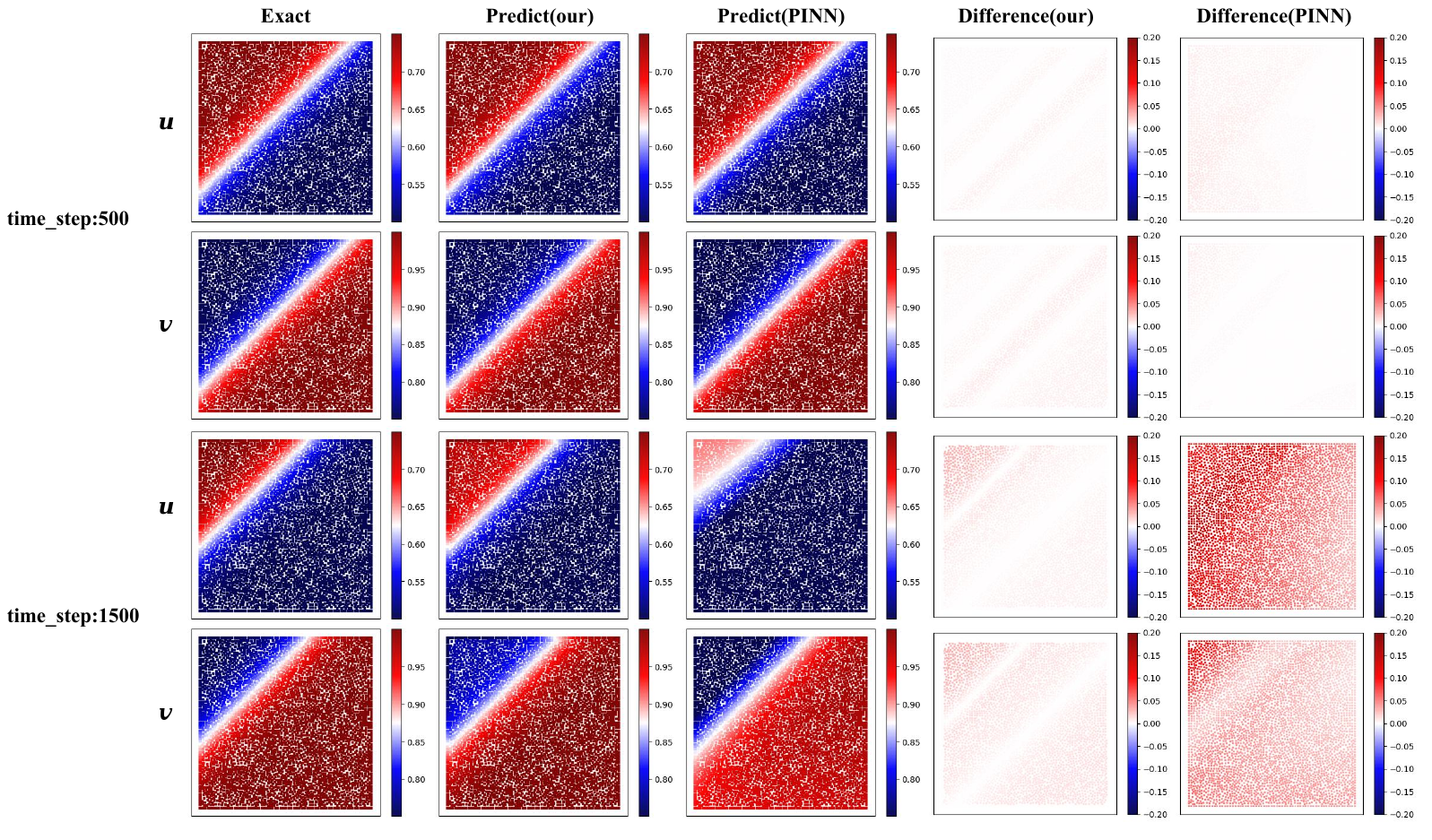}}
\caption{\footnotesize{Generalization performances of heat (top) and Burgers (bottom) equations under the default computational domain with coarser spatial resolution. Computational domain: $[0,1]\times[0,1]$, mesh density=50 (100 by default). }}
\label{fig9}
\end{figure}

\begin{enumerate}
\smallskip
\item [\quad\textbf{(3)}] \textbf{Enlarging computational domain}
\smallskip
\end{enumerate}

Next, we expand the computational domain by 4 times, keeping the same mesh size (spatial resolution), that is, the computational domain is enlarged from $[0,1]\times[0,1]$ to $[0,2]\times[0,2]$, and the mesh density is increased from 100 to 400. We predict solutions of heat and Burgers equations for 1000 and 1500 steps, and visualize the rollout performance in Fig. \ref{fig10}, respectively. The top subplot shows that for heat equation PINN is extremely sensitive to changes in computational domain and outputs completely wrong results, while predictions of PIGNN are consistent with the analytical reference. And from the bottom subplot of Fig. \ref{fig10}, we see that our PIGNN has smaller errors than PINN at all time steps for Burgers equation. This experiment demonstrates that our PIGNN is also robust to changes in computational domain.

\begin{figure}[!htbp]
\centering
\subfigure[heat]{\includegraphics[width=0.95\linewidth]{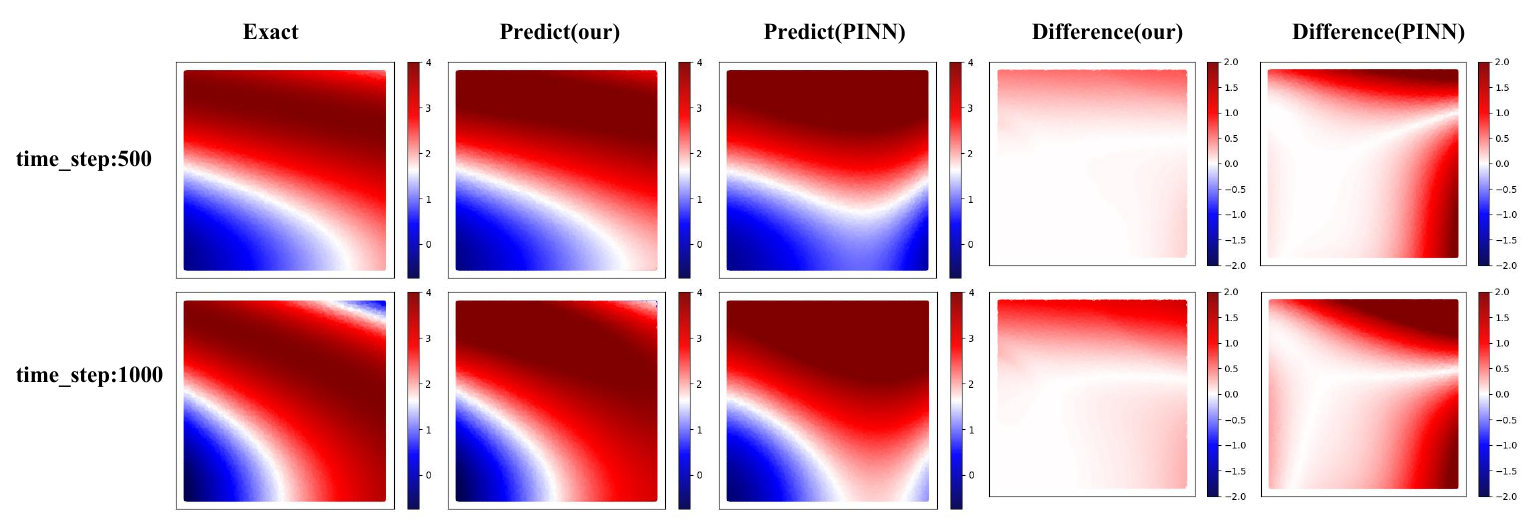}}
\subfigure[Burgers]{\includegraphics[width=0.95\linewidth]{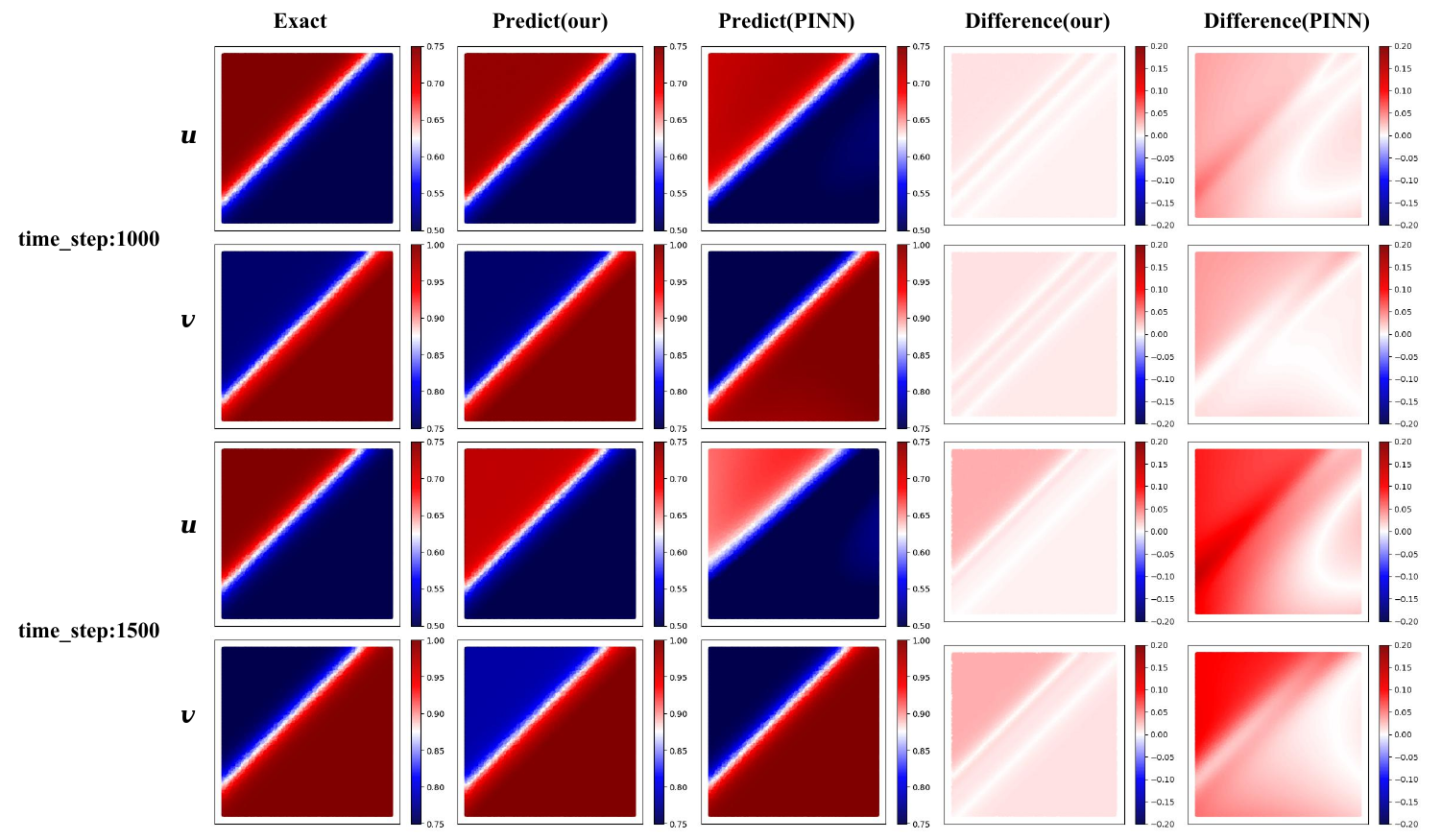}}
\caption{\footnotesize{Generalization performances of heat (top) and Burgers (bottom) equations under the enlarged computational domain with the same spatial resolution. Computational domain: $[0,2]\times[0,2]$, mesh density=400. Default settings: computational domain: $[0,1]\times[0,1]$, mesh density=100.}}
\label{fig10}
\end{figure}

\begin{enumerate}
\smallskip
\item [\quad\textbf{(4)}] \textbf{Enlarged computational domain with coarser spatial resolution}
\smallskip
\end{enumerate}

Furthermore, we expand the computational domain by 9 times, i.e., it is enlarged from $[0,1]\times[0,1]$ to $[0,3]\times[0,3]$, and let mesh density be 300 (100 by default). As a result, the spatial resolution becomes much coarser than default settings. Same as the previous experiment, we obtain 1000 and 1500 time steps predicted solutions of heat and burgers equations respectively and describe two time steps snapshots in Fig. \ref{fig11}. Similar to previous results, our PIGNN presents outstanding fitness to analytical solutions of both equations, while PINN can simulate part of dynamical patterns of the Burgers equation, but hardly in the heat equation.

\begin{figure}[!htbp]
\centering
\subfigure[heat]{\includegraphics[width=0.95\linewidth]{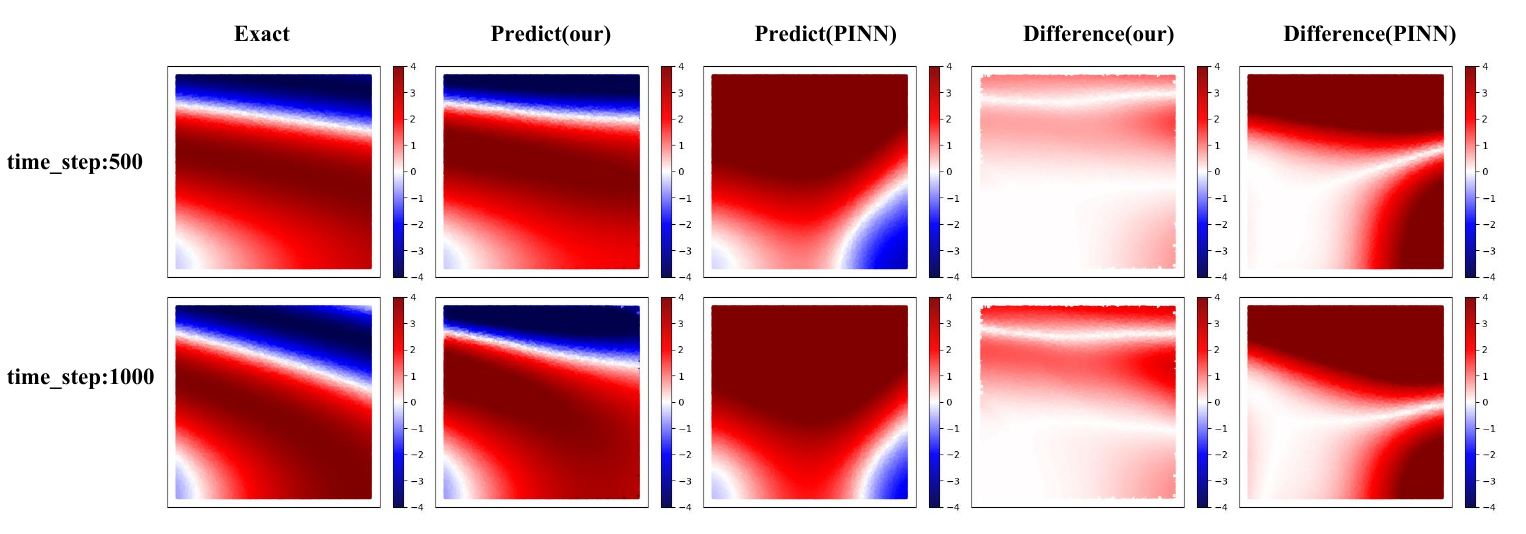}}
\subfigure[Burgers]{\includegraphics[width=0.95\linewidth]{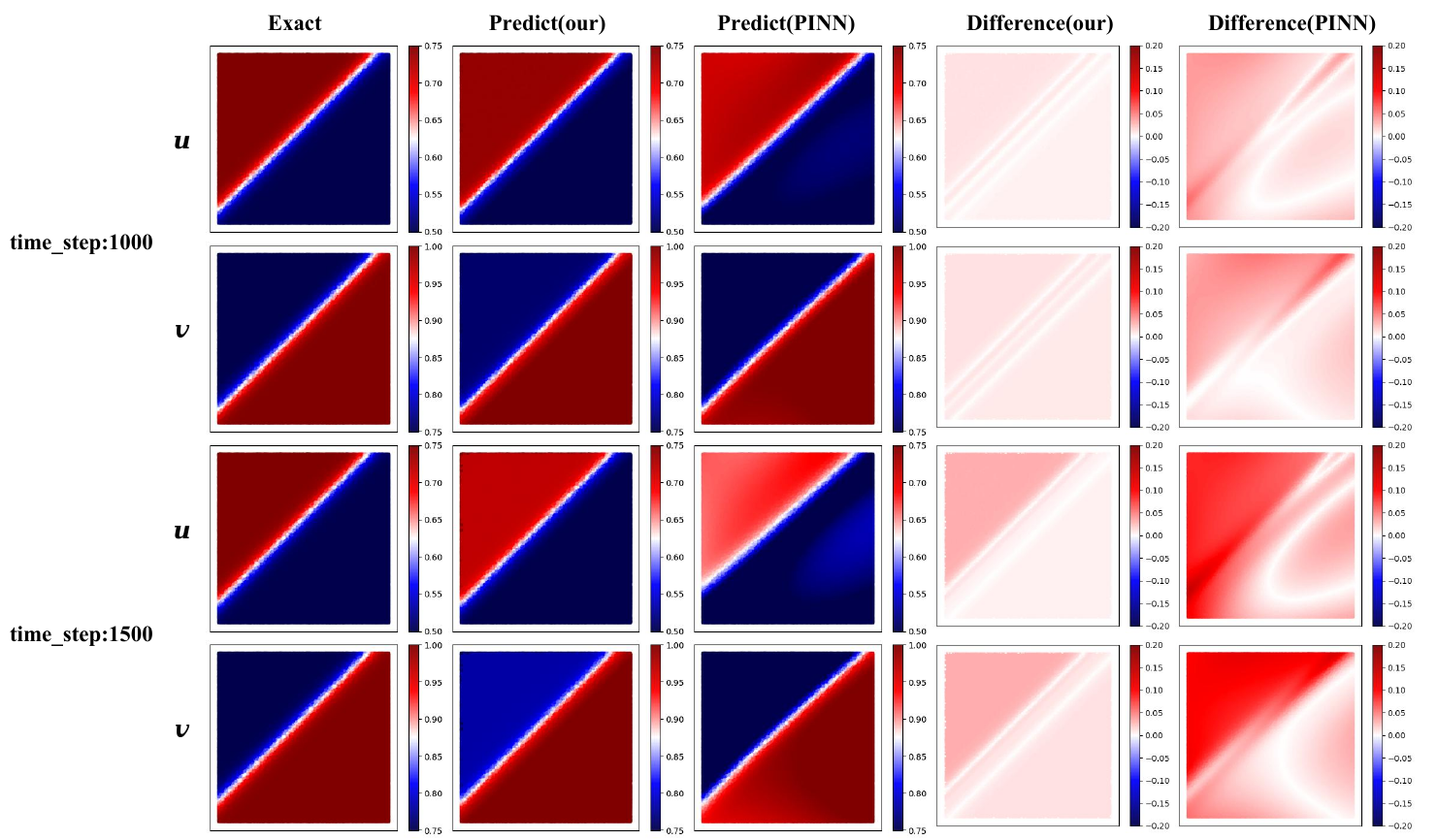}}
\caption{\footnotesize{Generalization performances of heat (top) and Burgers (bottom) equations under the enlarged computational domain with coarser spatial resolution. Computational domain: $[0,3]\times[0,3]$, mesh density=300. Default settings: computational domain: $[0,1]\times[0,1]$, mesh density=100.}}
\label{fig11}
\end{figure}

\begin{enumerate}
\smallskip
\item [\quad\textbf{(5)}] \textbf{Enlarged computational domain with finer spatial resolution}
\smallskip
\end{enumerate}

Moreover, we expand the computational domain by 2 times, i.e., it is enlarged from $[0,1]\times[0,1]$ to $[0,2]\times[0,1]$, and let mesh density be 400 (100 by default). Namely, the spatial resolution is finer than default settings. Fig. \ref{fig12} shows the visual predictions of PIGNN and PINN for heat and Burgers equations in comparison with the reference. As in previous three experiments, results in this case show similar patterns, and we don't go into details here.

In addition, in Fig. \ref{fig13}, we depict the aRMSE curves of heat (top) and Burgers (bottom) equations separately for all spatial resolution cases in enlarged computational domains. Let's first explain the meaning of the legend in the figure. `initial' represents the case with default settings, `remain' means the case in the enlarged computational domain ($[0,2]\times[0,2]$) with the same mesh size (density is 400), `coarser' corresponds to the case in the enlarged computational domain ($[0,3]\times[0,3]$) with large mesh size (density is 300), and `finer' refers to the case in the enlarged computational domain ($[0,2]\times[0,1]$) with small mesh size (density is 400). Dashed lines represent the error curves of PINN, while solid lines are curves of PIGNN. As can be seen from the above subgraph, for each case of the heat equation, the error curve of PIGNN is always lower than the corresponding PINN curve. Comparing all the PIGNN error curves, the default case has the smallest error, while the 'coarser' case possesses the largest error. However, from the subgraph below we see that for the Burgers equation, the phenomenon is a little different. For all PIGNN cases, the error curves do not vary too much and are concentrated around the default error curve. While for PINN cases, the default error curve is lower than that of PIGNN at first and then increases dramatically. After half of the time step, all PINN error curves increase rapidly, but PIGNN's error curves tend to be flat steadily. The results in this figure further validate the robustness of PIGNN to changes in computational domain and spatial resolution.

\begin{figure}[!htbp]
\centering
\subfigure[heat]{\includegraphics[width=0.95\linewidth]{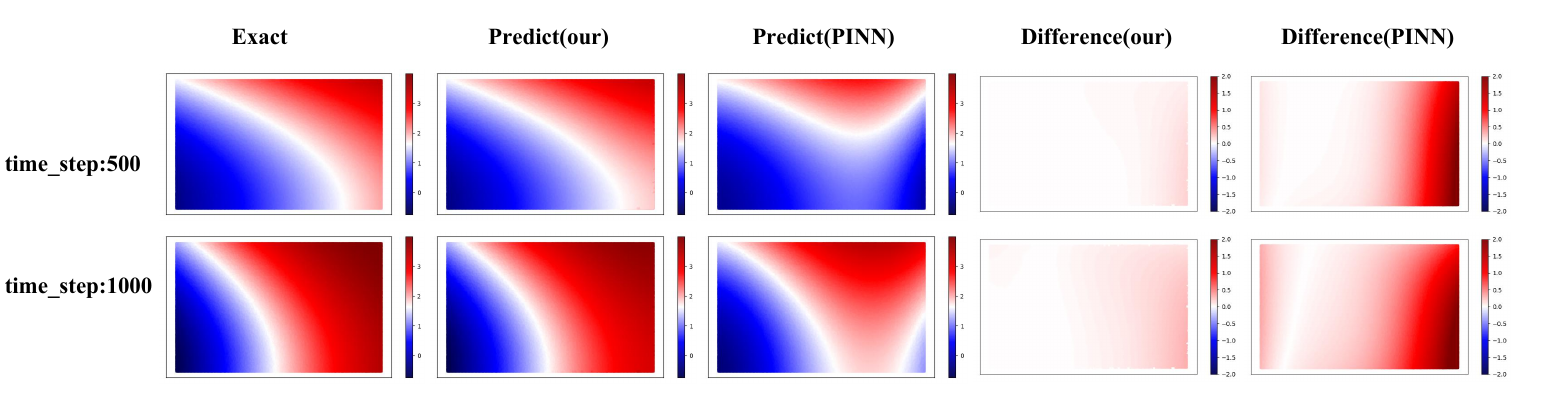}}
\subfigure[Burgers]{\includegraphics[width=0.95\linewidth]{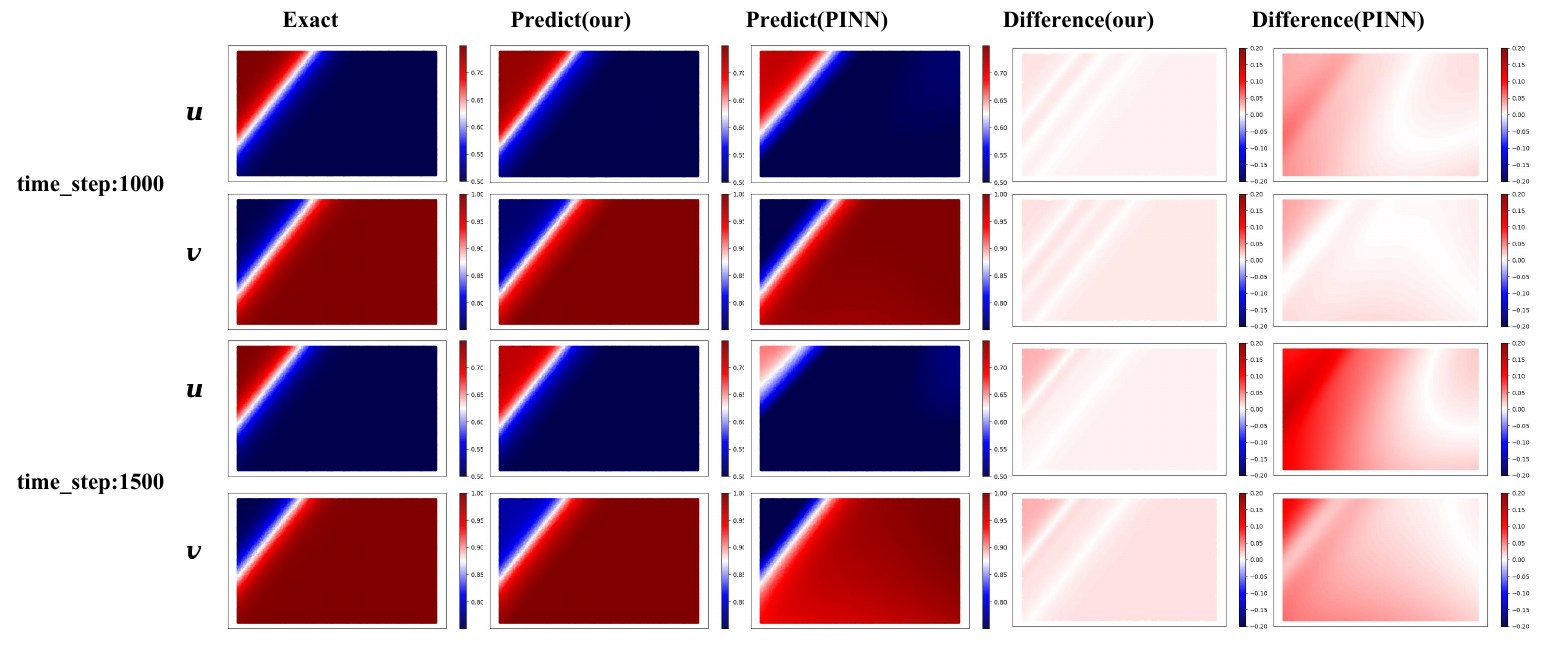}}
\caption{\footnotesize{Generalization performances of heat (top) and Burgers (bottom) equations under the enlarged computational domain with finer spatial resolution. Computational domain: $[0,2]\times[0,1]$, mesh density=400. Default settings: computational domain: $[0,1]\times[0,1]$, mesh density=100.}}
\label{fig12}
\end{figure}

\begin{figure}[!htbp]
\centering
\subfigure{\includegraphics[width=0.495\linewidth]{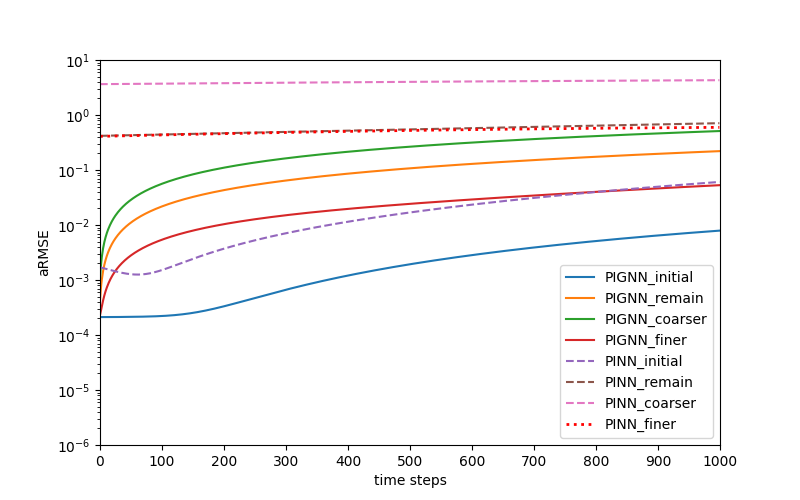}}
\subfigure{\includegraphics[width=0.495\linewidth]{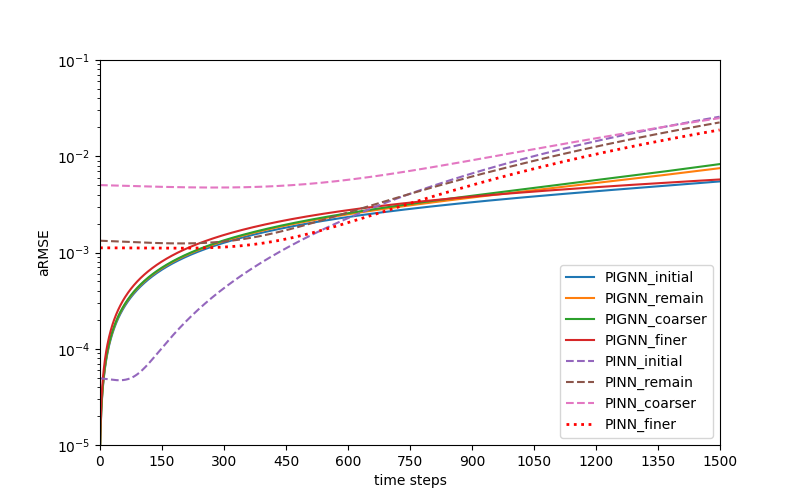}}
\caption{\footnotesize{Comparison of accumulative error curves for heat (left) and Burgers (right) equations under different computational domains with different spatial resolutions.}}
\label{fig13}
\end{figure}

\begin{enumerate}
\smallskip
\item [\quad\textbf{(6)}] \textbf{Changing computational domain, spatial resolution and PDE parameters simultaneously}
\smallskip
\end{enumerate}

In the last part of this subsection, we expand the computational domain by 2 times with finer spatial resolution for heat and Burgers equations. That is, the computational domain is enlarged from $[0,1]\times[0,1]$ to $[0,1]\times[0,2]$ and the mesh density is increased from 100 to 300. We set $b=6$ (2 by default) and $R_e=60$ (80 by default) in the heat and Burgers equations, respectively. For the FN equation, to further demonstrate the fitting ability of PIGNN under complex settings in large domains, we train the PIGNN model with four groups of parameters at the default settings, and then change the IC with the following complex expression when testing
\begin{eqnarray*}
u &=& \sin[200\pi(x+y)+0.9]+1.5\sin[200\pi xy(x-y)] + 1.4\cos(\pi y),\\
v &=& 0.9\sin[200\pi xy(x-0.5)(x-y)+0.9] + 1.1\sin[50\pi(x-0.5)y + 1.4].
\end{eqnarray*}
Meantime, we expand the computational domain by 4 times, $[-1,1]\times[-1,1]$ and increase the mesh density to 1000.

All predictions are pictured in Figs. \ref{fig14} and \ref{fig15}. It appears from Fig. \ref{fig14} that, even when we change the PDE parameter, computational domain and spatial resolution simultaneously, our PIGNN still gives satisfactory predictions for both heat and Burgers equations. PINN does not perform well in the heat equation, but produces relatively better results in the Burgers equation. Frankly speaking, the error in the current case is larger than the error when changing the computational domain, spatial resolution or PDE parameter separately. Fig. \ref{fig15} shows that PIGNN is able to infer the high-frequency dynamic patterns never seen during training, suggesting that our model has the potential to generalize to more complex situations.

We think the above prominent strength comes from the architecture design of our graph network and appropriate choices of node features, which is conductive to learning local dynamic patterns of systems. Furthermore, by forcing the network to make next-step prediction on irregular meshes with distance-relative edge features, our model is encouraged to learn resolution-independent physics. In order to clearly show the results of FN equation with high-frequency dynamics, we provide an amplified version of result $\bm u$ at time step 80. It should be remarked here that since the CPU resources are already exhausted when the numerical software is working with a mesh density of 1000, we have no reference results in this case. This experiment also demonstrates the availability and substitutability of our PIGNN compared with numerical softwares, that is, our model is an alternative tool when the numerical software cannot perform calculations due to the large amount of computations on the same computer configuration.

\begin{figure}[!htbp]
\centering
\subfigure[heat]{\includegraphics[width=0.95\linewidth]{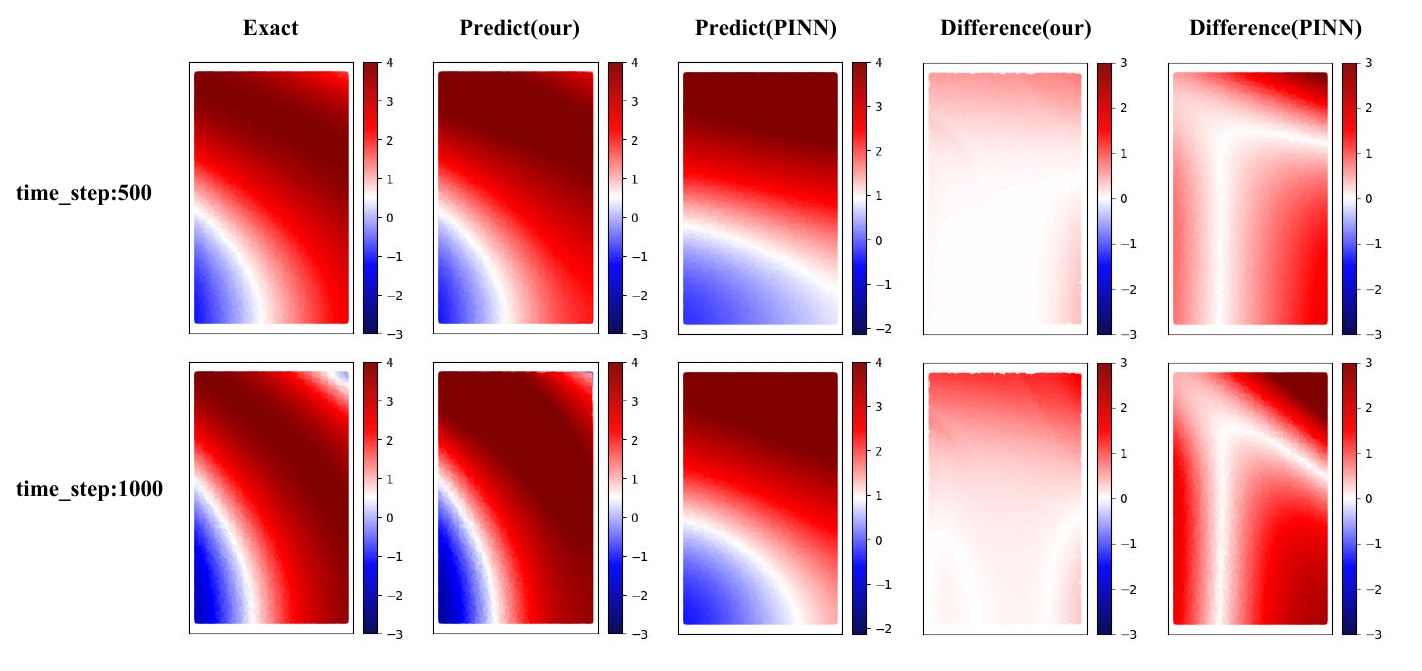}}
\subfigure[Burgers]{\includegraphics[width=0.95\linewidth]{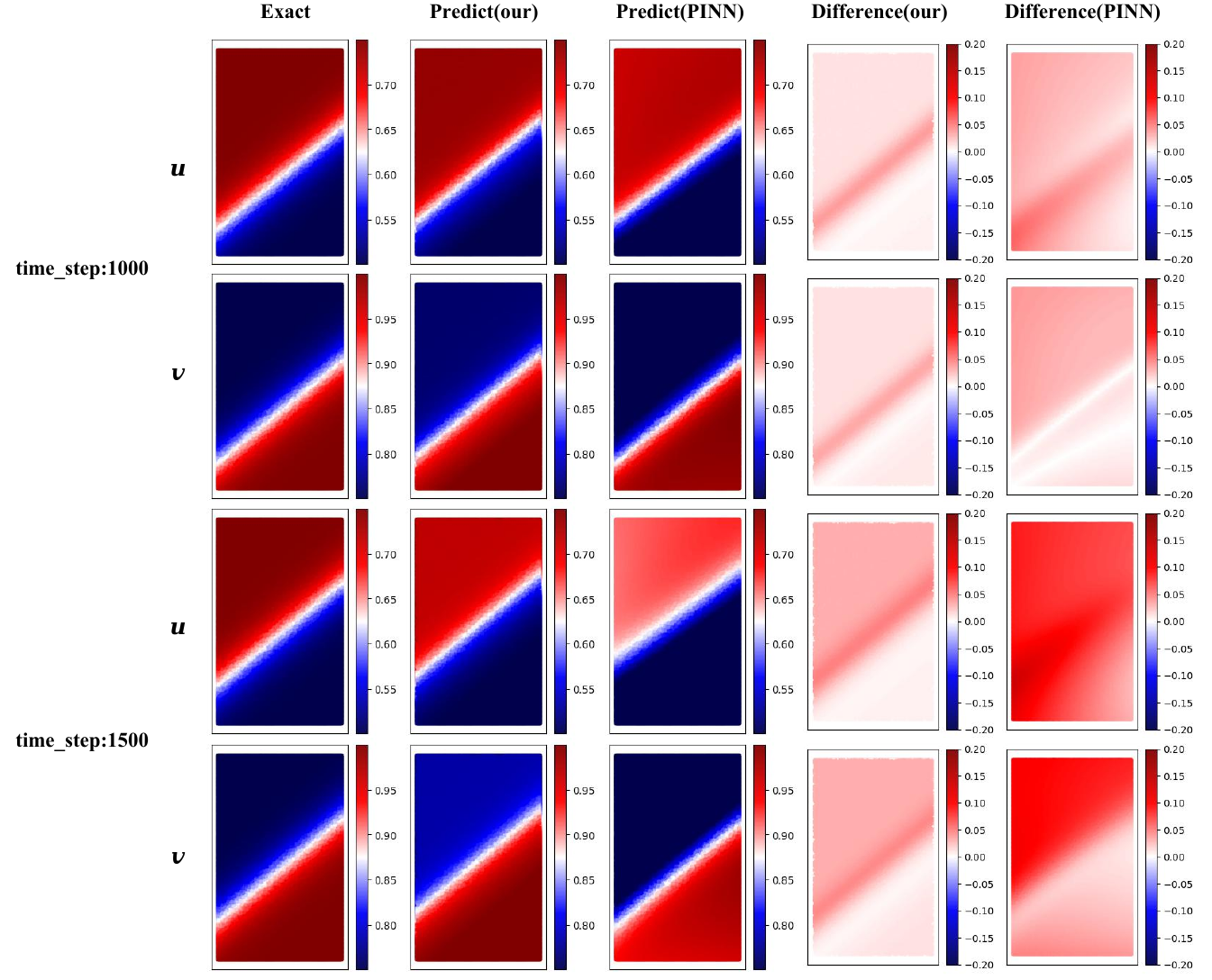}}
\caption{\footnotesize{Generalization performance of heat (above) and Burgers (below) equations with different computational domain, spatial resolution and PDE parameter. heat: $b=6$, Burgers: $R_e=60$. Computation domain: $[0,1]\times[0,2]$, mesh density=300. Default settings: computational domain: $[0,1]\times[0,1]$, mesh density=100, $b=2$, $R_e=80$.}}
\label{fig14}
\end{figure}

\begin{figure}[!htbp]
\begin{center}
\scalebox{0.52}[0.52]{\includegraphics{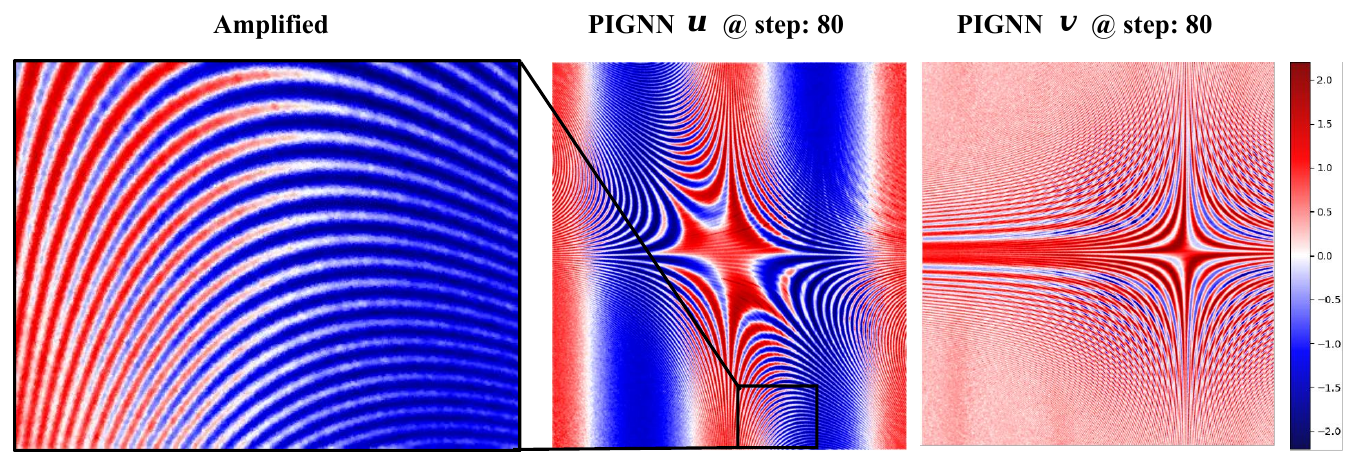}}
\caption{\footnotesize{Generalization performance of the FN equation with different computational domain, spatial resolution and IC. Computation domain: $[-1,1]\times[-1,1]$, mesh density=1000. Default settings: computational domain: $[0,1]\times[0,1]$, mesh density=100.}}
\label{fig15}
\end{center}
\end{figure}

\subsubsection{Scalability on the forward problem}
\noindent{\textbf{\\}}

In this subsection, taking the Burgers equation as an example, we present parallel results of a large-scale PIGNN for distributed multi-GPU systems to demonstrate the scalability of our PIGNN. Since the focus of this paper is not the distributed parallelism of large-scale graph neural networks, technical details of parallel computing will not be introduced here, and we only show the results. From an implementation point of view, the parallelism of the large-scale PIGNN is simpler than other parallel algorithms because the model structure remains the same. Only the graph partition, the exchange and synchronization of boundary information of each subgraph need to be considered.

In this experiment, the computational domain is $[0,28]\times[0,28]$, and the mesh density is 2800, resulting in a total number of nodes of 12170332. The graph is divided into four subgraphs so that the number of nodes in each subgraph is as close as possible. Fig. \ref{fig16} shows snapshots of solutions and differences for the large-scale problem in the training and extrapolation phase at time steps 5 and 1500. In order to clearly show the results of such high resolution, we provide an amplified version of result $\bm v$ at time step 5 in Fig. \ref{fig17}. As can be seen from Fig. \ref{fig16} that the error is small even at very long time steps. This experiment verifies that our PIGNN can be well scaled to distributed systems to deal with large-scale problems.

\begin{figure}[!htbp]
\begin{center}
\scalebox{0.45}[0.45]{\includegraphics{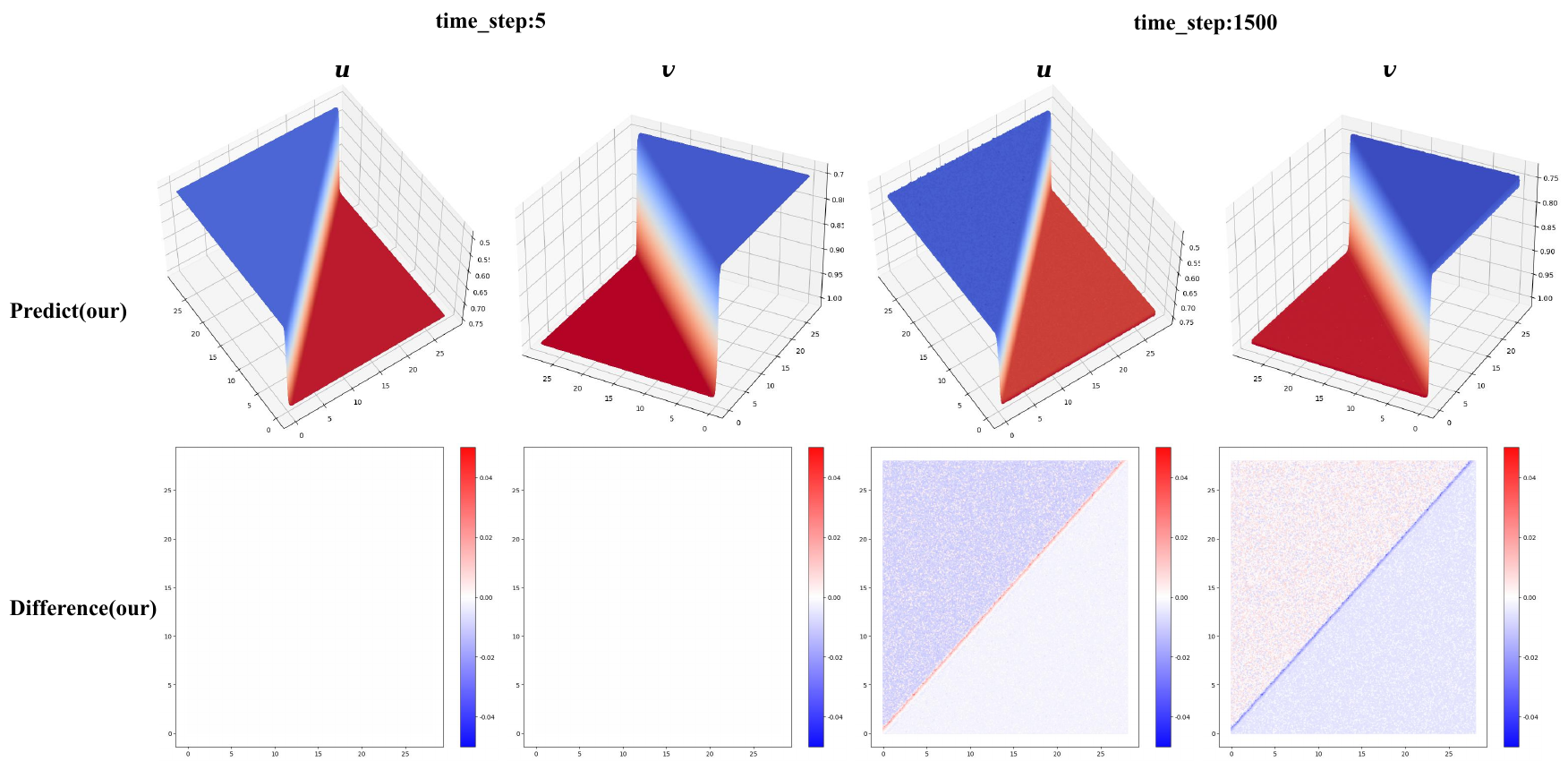}}
\caption{\footnotesize{Generalization performance of large-scale PIGNN for Burgers equation. Computation domain: $[0,28]\times[0,28]$, mesh density=2800. Default settings: computational domain: $[0,1]\times[0,1]$, mesh density=100.}}
\label{fig16}
\end{center}
\end{figure}

\begin{figure}[!htbp]
\begin{center}
\scalebox{0.45}[0.45]{\includegraphics{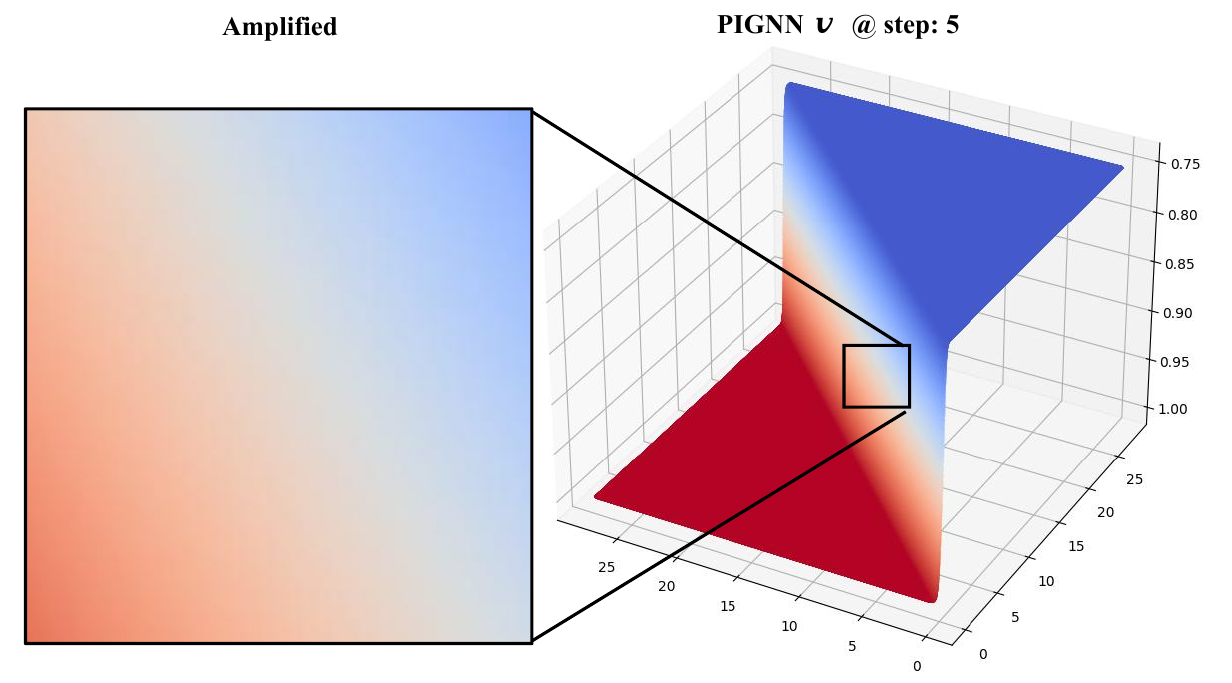}}
\caption{\footnotesize{The amplified result $\bm v$ of large-scale PIGNN at time step 5 for the Burgers equation.}}
\label{fig17}
\end{center}
\end{figure}

\subsubsection{Inverse problem}
\noindent{\textbf{\\}}

For the inverse problem, we take a simple heat equation as an example for explanation. We choose the following equation,
\begin{eqnarray}\label{Eq5.1}
\frac{\partial u(x,y,t)}{\partial t} &=& \Delta u(x,y,t) + f(x,y,t), \\
f(x,y,t) &=& -kx\sin[ktx+y] + \cos(ktx+y)(k^2t^2+1), \nonumber\\
(x,y)\in\Omega &=& [0,1] \times [0,1],\,\, t>0, \nonumber
\end{eqnarray}

subject to the initial and boundary conditions
\begin{eqnarray*}
u(x,y,0) &=& \cos(y), \qquad\qquad\qquad\,\,\, (x,y)\in\Omega,\\
u(x,y,t) &=& \cos(4tx+y),\quad (x,y)\in\partial \Omega,\, t>0.
\end{eqnarray*}

The analytical solution is
\begin{equation}\label{Eq4.4}
u(x,y,t) = \cos(ktx+y).
\end{equation}

To highlight the ability of PIGNN of learning from scarce training data, we randomly sample 100 points from the computational domain, accounting for only 0.67\% of total nodes, and select the corresponding analytical solutions as observations to form the training dataset. Fig. \ref{fig18} shows the training process of our PIGNN under three different initial values of the unknown parameter $k$. It is clear from the figure that for arbitrary initial value, all curves converge to the same value as the model training process progresses, which implies that our PIGNN can identify the true parameter.

\begin{figure}[!htbp]
\begin{center}
\scalebox{0.52}[0.52]{\includegraphics{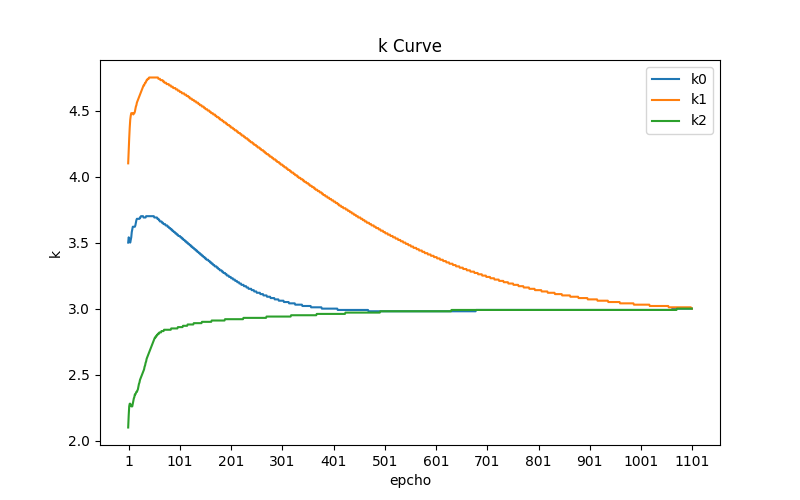}}
\caption{\footnotesize{Training process of PIGNN under different initial values of unknown parameter of heat equation.}}
\label{fig18}
\end{center}
\end{figure}

\subsubsection{Ablation study}
\noindent{\textbf{\\}}

In this subsection, two additional experiments are carried out to demonstrate the influences of the number of GN blocks and the choice of node features on model's performance.

\begin{enumerate}
\smallskip
\item [\quad\textbf{(1)}] \textbf{Effect of the number of GN blocks}
\smallskip
\end{enumerate}

In this experiment, we choose three sets of heat equation parameters $b=-2,0.9,8$, and test the performance of our PIGNN under one and two GN blocks, respectively. Fig. \ref{fig19} describes the aRMSE curve with respect to time steps for each case. It can be seen that for the same parameter value, the error curve of one GN block is consistently lower than that of two GN blocks, but the gap is quite small. We think this is due to the fact that each node is only sensitive to changes of its one-ring neighbours, and model with one GN block is sufficient enough to learn such dynamic patterns.

\begin{figure}[!htbp]
\begin{minipage}[t]{0.485\linewidth}
\centerline{\includegraphics[width=1.15\textwidth]{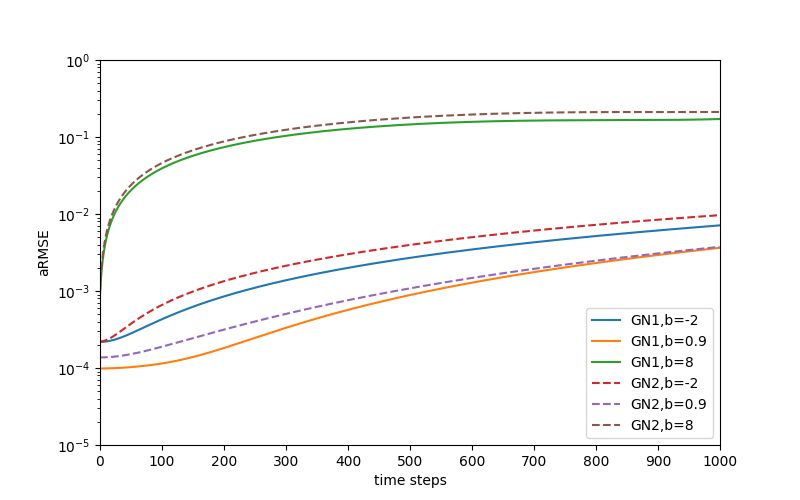}}
\caption{\footnotesize{ Generalization performance of heat equation with different number of GN blocks and PDE parameters. }}
\label{fig19}
\end{minipage}
\hfill
\begin{minipage}[t]{0.485\linewidth}
\centerline{\includegraphics[width=1.15\textwidth]{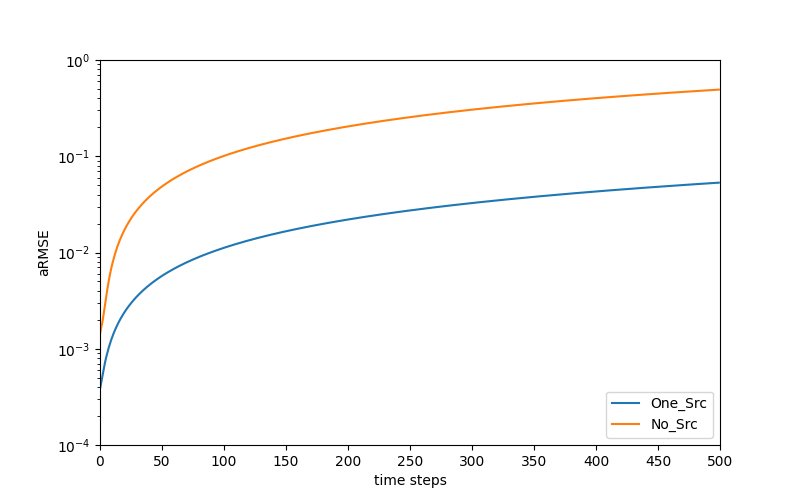}}
\caption{\footnotesize{Comparison of aRMSE curves of heat equation with different choices of node features.}}
\label{fig20}
\end{minipage}
\end{figure}

\begin{enumerate}
\smallskip
\item [\quad\textbf{(2)}] \textbf{Importance of appropriate node features}
\smallskip
\end{enumerate}

Here, we are going to demonstrate the importance of choosing appropriate node features. As we explained in subsection \ref{subsec2.4}, if the source term of equation is nonzero, we have to add it to node features to capture the spatial dynamic pattern. We test this selecting rule by choosing two sets of node features. The first case refers to the default setting, that is, node features contain node solutions and node type vectors. In the second case, besides the default setting, the source term $f(t,x,y)$ in equation \eqref{Eq4.3} is added to node features. We term these two cases: No\_src and One\_src. The aRMSE curves for both cases are presented in Fig. \ref{fig20}. Furthermore, we provide two representative snapshots in Fig. \ref{fig21} to visualize the impact of different node features on PIGNN's predictions. Figure \ref{fig20} clearly shows that the error of No\_src is higher than that of the other case. The results of Figure \ref{fig21} further confirm this conclusion, where the results of No\_src show that the model without additional source feature rarely learns the dynamics of the system, while the model with One\_src has successfully learned.

\begin{figure}[!htbp]
\begin{center}
\scalebox{0.56}[0.56]{\includegraphics{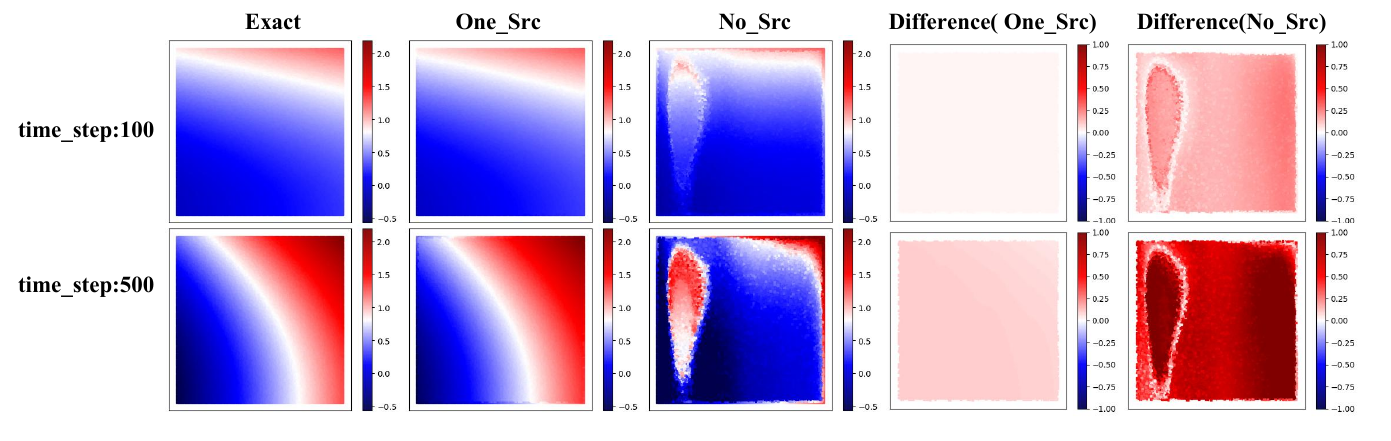}}
\caption{\footnotesize{Generalization performance of the heat equation with different choices of node features.}}
\label{fig21}
\end{center}
\end{figure}

In the final part of this article, we make an overall comparison between our PIGNN and PINN from 10 aspects to further demonstrate the outstanding properties of our PIGNN. PIGNN can accurately model various PDE systems with good time extrapolability and can be well generalized to different computational domains, spatial resolutions (mesh size) and PDE parameters simultaneously. For details, please see Table \ref{tab2}.

\begin{small}
\begin{table}[!htbp]
\renewcommand\arraystretch{1.3}
\def\baselinestretch{1.0}
\caption{General comparison between PIGNN and PINN.}
\label{tab2}
\centering
\resizebox{\textwidth}{!}{
\begin{tabular*}{1.0\textwidth}{@{\extracolsep{\fill}}lcccc}
\toprule
                              & PIGNN                & PINN              \\
\midrule
number of training samples    &  0                   &    0               \\ \hline
mesh                          &  regular/irregular   &  meshless         \\ \hline
mesh size                     &    flexible          &    ---             \\ \hline
problem scale                 &    big               &    small           \\ \hline
inference error               &   $10^{-2}$          &    $10^{-2}$           \\ \hline
\multirow{2}{*}{extension of computational domain} &   extend significantly    &    \multirow{2}{*}{extend moderately}       \\
 & to huge domain   &    \\\hline
extension of time steps       &   extend plenty of time steps   &    ---      \\ \hline
change PDE parameters         &   YES   &    NO       \\ \hline
inverse problem               &   YES   &    YES       \\\hline
\multirow{2}{*}{conclusion}                   &   Small area training,    &  training area $\approx$       \\
&  large area application.  &  application area \\
\bottomrule
\end{tabular*}}
\end{table}
\end{small}

\section{Conclusion}\label{sec5}
\noindent{\textbf{\\}}

In this paper, we propose a novel PIGNN framework for solving both the forward and inverse problems governed by spatiotemporal PDEs. Thanks to the integration of graph neural networks and physics-informed loss function, the proposed PIGNN can naturally handle irregular meshes that can fit any geometry, as well as learn dynamics of the underlying system. We demonstrate the effectiveness and excellent performance of our PIGNN in terms of accuracy, time extrapolability and generalizability through a series of numerical experiments. Promising numerical results show that PIGNN outperforms the neural-based solver PINN. Most importantly, PIGNN only needs to be trained in small domains under simple settings to have excellent fitting ability, and then it can be directly applied to more complex situations with large domains, which further sheds lights on the superiority of the proposed model. Moreover, we show that PIGNN can scale well to distributed system to handle large-scale problems. The outstanding performance of PIGNN suggests that it has promising potential to serve as a reliable surrogate model to learn physical dynamics described by PDEs or as a neural-based solver to infer solutions for PDEs. Furthermore, our work reveals that integrating deep learning techniques, scientific knowledge and numerical methods together can facilitate the entire modelling and learning process of complex systems.

%

\section*{Acknowledgements}
\noindent{\textbf{\\}}

Zhang Hao laid the mathematical foundation of the project, designed the main algorithms and supervised the entire project. Jiang Longxiang implemented, trained and tested most of the code and conducted some experiments. Wang Liyuan wrote and revised the whole manuscript, optimized the training process and completed most of the numerical experiments. Chu Xinkun revised the writing of the whole manuscript and provided support to the design of some experiments. Wen Yong and Li Luxiong helped to do some numerical experiments. Xiao Yonghao helped with arrangements of the project. Thanks for all the efforts we have done to accomplish this paper. The present research was supported by Sichuan Province Science Foundation for Youths under Grant Number 2023NSFSC1419.

\section*{Appendix}

\subsection*{A. Proof of Lemma 3.2.}
\begin{proof}
Let $l_i$ be the order of $\varphi_i\,(i=1,\cdots,n)$, i.e., $\varphi_i=x^{s}y^{l_i-s},\,\,{\rm for}\,\,{\rm some}\,s\in[0,l_i]$, then we have
\begin{equation}\label{EqB.1}
\varphi_i(r\bm{{\rm x}}_k)=r^{l_i}\varphi_i(\bm{{\rm x}}_k),\,\forall \,\,k\in1,\cdots,p.
\end{equation}

Partitioning matrix $\bm{\Phi}$ according to the order of $\varphi_i$, i.e., $\bm{\Phi}=\big(\bm{\Phi}_{1}^{\rm T},\cdots, \bm{\Phi}_{m}^{\rm T}\big)^{\rm T}$,
where $\bm{\Phi}_{j}\,(j\in1,\cdots,m)$ is the submatrix of $\bm{\Phi}$ formed by taking all rows of entries $\varphi_i$ with order of $j$, partitioning vector $\bm{{\rm Y}}$ accordingly. By the definition of $\bm{{\rm Y}}$ and $\varphi_i$, $\bm{Y}_{i}\neq\bm{0}$ only for $i=2$.

Let $\bm{\eta}=r^2\bm{\omega}$, by Eq. \eqref{EqB.1}, \eqref{Eq3.9.4} can be written as
\begin{equation}\label{EqB.2}
\bm{{\rm Y}}=\bm{\Lambda}_r\bm{\Phi}\bm{\eta},
\end{equation}
where $\bm{\Lambda}_r = {\rm diag}(r^{-1},1,r,\cdots, r^{m-2})$.

Let $\bm{\eta}_r\in\mathbb R^p$ be the least square solution of Eq. \eqref{EqB.2}.

1) Let us first study that $\bm{\eta}_r$ is bounded.

By Eq. \eqref{EqB.2},
\begin{equation}\label{EqB.3}
\|\bm{{\rm Y}}-\bm{\Lambda}_r\bm{\Phi}\bm{\eta}\|^2 =
r^{-2} \|\bm{\Phi}_{1}{\bm\eta}\|^2
+ \|\bm{{\rm Y}}_2-\bm{\Phi}_{2}{\bm\eta} \|^2
+ \cdots
+ r^{2(m-2)}\|\bm{\Phi}_{m}{\bm\eta} \|^2.
\end{equation}
\quad Suppose $\bm{\eta}_0$ is an approximation of Laplacian $\Delta$ with the largest order of $l$. According to definition \ref{def1}, $\bm{{\rm Y}}_i-\bm{\Phi}_{i}\bm{\eta}_0=\bm{0}, \,\,{\rm for}\,\, i=1,2,\cdots,l, ,\,{\rm while}\,\,\bm{{\rm Y}}_{l+1}-\bm{\Phi}_{l+1}\bm{\eta}_0\neq \bm{0}.$

Since $\bm{\eta}_r $ is the least square solution of Eq. \eqref{EqB.2}, we have
\begin{equation}\label{EqB.3.1}
\|\bm{{\rm Y}}-\bm{\Lambda}_r\bm{\Phi}\bm{\eta}_r\|^2\leq
\|\bm{{\rm Y}}-\bm{\Lambda}_r\bm{\Phi}\bm{\eta}_0\|^2,
\end{equation}

It follows from \eqref{EqB.3} and \eqref{EqB.3.1} that
\begin{equation*}
r^{2(i-2)}\|\bm{\Phi}_{i}{\bm\eta}_r \|^2 \leq
r^{2(l-1)}\|\bm{\Phi}_{l+1}{\bm\eta}_0 \|^2
+ \cdots
+ r^{2(m-2)}\|\bm{\Phi}_{m}{\bm\eta} \|^2, \forall \, i= 1,2,\cdots,l+1.
\end{equation*}

This implies there exists a constant $C>0$ such that for any small enough $r$, we have
\begin{equation}\label{EqB.4}
\|\bm{\Phi}_{i}{\bm\eta}_r \|^2 \leq
C\|\bm{\Phi}_{l+1}{\bm\eta}_0 \|^2, \,\,{\rm for}\,\, i=1,2,\cdots,l+1.
\end{equation}

Let $\bm{\Phi}_{[l]} = \big[\bm{\Phi}_{1}^{\rm T}, \cdots, \bm{\Phi}_{l}^{\rm T}\big]^{\rm T}$, $\bm{Y}_{[l]} =[\bm{Y}_{1}^{\rm T}, \cdots, \bm{Y}_{l}^{\rm T}]^{\rm T}$. Since $\bm{\Phi}_{[l]}{\bm\eta} = \bm{Y}_{[l]}$ admits a solution, $\bm{\Phi}_{[l+1]}{\bm\eta} = \bm{Y}_{[l+1]}$ has no solution, $\bm{\Phi}_{[l+1]}$ is a full rank matrix.

Select the submatrix $\tilde{\bm{\Phi}}$ from $\bm{\Phi}_{[l+1]}$ such that $\bm{\tilde{\Phi}}$ is a nonsingular square matrix, and select the corresponding rows from $\bm{Y}_{[l+1]}$ to form $\bm{\tilde{Y}}$. It easily follows that ${\bm\eta}_r=\bm{\tilde{\Phi}}^{-1} \bm{\tilde{Y}}$.

Thus
\begin{equation}\label{EqB.5}
\|{\bm\eta}_r\|\leq\|\bm{\tilde{\Phi}}^{-1}\|\| \bm{\tilde{Y}}\|
\end{equation}

By Eq. \eqref{EqB.4}, $\|\bm{Y}_{[l+1]}\|= \|\bm{\Phi}_{[l+1]}{\bm\eta}_r\|$ is bounded, which implies the right-hand side of Eq. \eqref{EqB.5} is bounded.  Hence, assertion 1) follows.

\medskip
2) Next, let $\bm{\eta}_n=\bm{\eta}_{r_n}$, and $\{\bm{\eta}_n\}_{n=1}^{\infty}$ be any convergent subsequence of $\bm{\eta}_{r}$, we examine that if $\bm{\eta}=\lim\limits_{n\rightarrow\infty} {\bm\eta}_{n}$, then $\bm{\eta}$ is an approximation of Laplacian $\Delta$ with the largest order.

We argue this by contradiction. Let $\bm{\eta}_0$ be an approximation of Laplacian $\Delta$ with the largest order of $l$, suppose $\bm{\eta}$ is an approximation of $\Delta$ with order of $s\,(s<l)$.

By \eqref{EqB.3}, we have
\begin{eqnarray}
\|\bm{{\rm Y}}-\bm{\Lambda}_r\bm{\Phi}\bm{\eta}\|^2 &=&
r^{2(s-1)}\|\bm{\Phi}_{s+1}{\bm\eta} \|^2
+ \cdots
+ r^{2(m-2)}\|\bm{\Phi}_{m}{\bm\eta} \|^2,  \label{EqB.6}\\
\|\bm{{\rm Y}}-\bm{\Lambda}_r\bm{\Phi}\bm{\eta}_0\|^2 &=&
 r^{2(l-1)}\|\bm{\Phi}_{l+1}{\bm\eta}_0 \|^2
+ \cdots
+ r^{2(m-2)}\|\bm{\Phi}_{m}{\bm\eta}_0 \|^2.  \label{EqB.7}
\end{eqnarray}
\quad Noting \eqref{EqB.6}, \eqref{EqB.7} and $\bm{\eta}=\lim\limits_{n\rightarrow\infty} {\bm\eta}_{n}$, there exists constants $\delta>0, \epsilon>0$ and $M>0$ such that
\begin{eqnarray*}
\|\bm{{\rm Y}}-\bm{\Lambda}_{r_n}\bm{\Phi}\bm{\eta}_n\|^2 \geq r_n^{2(s-1)}\|\bm{\Phi}_{s+1}{\bm\eta}_n \|^2
&>& \frac{r_n^{2(s-1)}}{2}\|\bm{\Phi}_{s+1}{\bm\eta} \|^2  \\
&\geq& \varepsilon r_n^{2(s-1)}
\geq Mr_n^{2(l-1)} \geq \|\bm{{\rm Y}}-\bm{\Lambda}_{r_n}\bm{\Phi}\bm{\eta}_0\|^2,\\
&&
\,\,{\rm for}\,\,n \,\,{\rm large}\,\,{\rm enough}.
\end{eqnarray*}
\quad This contradicts to the fact that $\bm{\eta}_n$ is the least square solution of Eq. \eqref{EqB.2}.

\end{proof}

\subsection*{B. Training setup}

\subsubsection*{B1. Setup of PIGNN}
\noindent{\textbf{\\}}

The specific settings for each component of our PIGNN are as follows. In ENCODER, two MLPs are used to encode node features and edge features, respectively. PROCESSOR is composed of $L$ GN blocks with the same structure, where $L$ is set to 1 by default. Two MLPs are used as update functions for internal nodes features and edges features. DECODER is also implemented with a MLP. All these MLPs contain two hidden layers, each with 128 rectified linear unit (ReLU) activation units. By default, we set the time interval $\Delta t=10^{-3}$, the chosen computational domain is $[0,1]\times[0,1]$ and is discretized by a $100\times 100$ irregular mesh, for a total of 15681 nodes. We use the python API of FEniCS to automatically create the irregular triangular mesh over the computational domain. Fig. \ref{fig22} shows an example of an irregular mesh on a regular domain $[0,1]\times[0,1]$ with a mesh width $\delta{\rm x}=0.1$ (or mesh density 10) generated by FEniCS. By default, our model is trained for only 10 time steps using the Adam (Adaptive moment estimation) optimizer with an exponential learning rate decay strategy from $10^{-3}$ to $10^{-7}$ every 10 epochs by 1\%.

\begin{figure}[htbp]
\begin{center}
\scalebox{0.4}[0.4]{\includegraphics{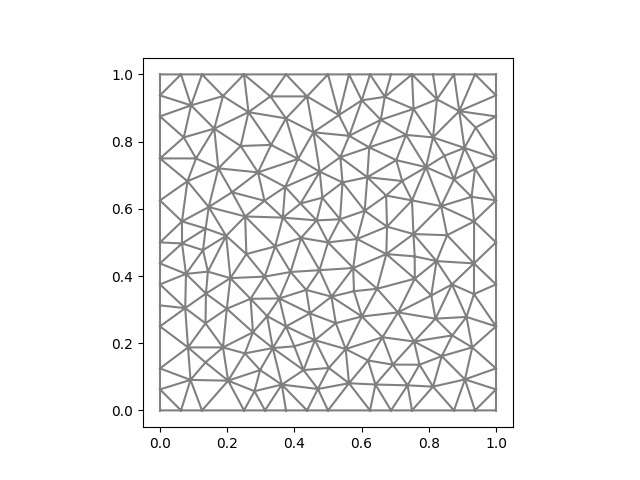}}
\caption{\footnotesize{An irregular mesh (mesh density is 10) on a regular domain $[0,1]\times[0,1]$ generated by FEniCS.}}
\label{fig22}
\end{center}
\end{figure}

\subsubsection*{B2. Setup of Baseline models}
\noindent{\textbf{\\}}

With the help of DEEPXDE \cite{Lu2021} library, we built the PINN baseline model. For a fair comparison, PINN's MLP architecture contains four hidden layers with 20 Tanh activation units each. The loss function is calculated by randomly selecting collocation points from both spatial and time spaces, where the number of collocation points is approximately equal to the number of mesh points multiplied by the total training time steps. When training, the computational domain and the time domain are the same as those in our approach. We train the model in two stages. In the first stage, the network is trained by the Adam optimizer with a learning rate of $10^{-3}$ for $6\times10^{4}$ epochs, and then in the second stage, it is trained by the L-BFGS (Limited-memory BFGS) optimizer for additional $4\times10^{4}$ epochs.

\subsubsection*{B3. Selection of hyper-parameters}
\noindent{\textbf{\\}}

An appropriate selection of hyper-parameters contributes to the success of model training. In this subsection, we introduce the criteria for selecting hyper-parameters.
\begin{itemize}
\item[$\bullet$] $\gamma$: This parameter balances the data loss and residual physics loss for network training in the inverse problem to help the network identify unknown PDE parameters. In practice, there is an order of magnitude difference between residual loss and data loss at the beginning of network training. Therefore, the initial value of $\gamma$ is set large (e.g., 0.9-0.95) to reduce the magnitude difference while allowing the network to mainly learn the dynamics of PDEs during the initial training period. As the training process proceeds, we gradually reduce $\gamma$ to 0.5 so that the network can learn more from data.
\smallskip

\item[$\bullet$] $\Delta t$: This parameter represents the time interval of the time space discretization. By default, we set it to $10^{-3}$. We use this setup in all experiments of Burgers and FN equations. For the heat equation, we set $\Delta t=10^{-4}$ since we are unable to successfully train the model using the default settings. Fortunately, numerical studies have provided us with a stability criterion for the spatialtemporal resolution. According to this criterion, the magnitude of $\Delta t$ is set to approximately the square of the spatial resolution, namely, $\Delta t \sim ({\delta x})^2$, where ${\delta x}$ is the mesh width which in our experiments is set to $10^{-2}$.
\smallskip

\item[$\bullet$] Training time steps: This parameter represents the number of time steps for network training. By default, we set it to 10, and use this setting in all experiments of Burgers and FN equations. For the heat equation, we set it to 50 because of its relatively small time interval. We set it larger to allow the model to learn more information over a longer evolutionary period.
\smallskip

\item[$\bullet$] Maximum number of training epochs: Our principle for setting the number of training epochs is that the model should be sufficiently trained. Specifically, for different equations, we adaptively adjust the concrete value according to the complexity of the problem. Generally, the number of training epochs for the Adam optimizer is about $6\times10^4 \sim 7\times10^4$, while for the L-BFGS optimizer, $3\times10^4\sim4\times10^4$ training epochs are used.
\end{itemize}

\subsection*{C. PDE examples}
\noindent{\textbf{\\}}

The specific expressions of PDE examples are as follows.

1) 2D Heat equation
\begin{eqnarray}\label{Eq4.3}
\frac{\partial u(x,y,t)}{\partial t} &=& \Delta u(x,y,t) + f(x,y,t), \\
f(x,y,t) &=& a\cos[b\pi(x-d)t + cy^2][b\pi(x-d) - 2c]  \nonumber \\
&&
\quad + a\sin[b\pi(x-d)t + cy^2][b^2\pi^2 t^2 + 4 c^2y^2], \nonumber\\
\Delta u(x,y,t) &=&  \frac{\partial^2 u(x,y,t)}{\partial x^2} + \frac{\partial^2 u(x,y,t)}{\partial y^2}, \nonumber\\
(x,y)\in\Omega &=& [0,1] \times [0,1],\,\, t>0, \nonumber
\end{eqnarray}

subject to the initial and boundary conditions
\begin{eqnarray*}
u(x,y,0) &=& a\sin(cy^2), \qquad\qquad\qquad\,\,\, (x,y)\in\Omega,\\
u(x,y,t) &=& a\sin[b\pi(x-d)t + cy^2],\quad (x,y)\in\partial \Omega,\, t>0.
\end{eqnarray*}

The analytical solution is
\begin{equation}\label{Eq4.4}
u(x,y,t) = a\sin[b\pi(x-d)t + cy^2],
\end{equation}

where $a,\,b,\,c,\,d = 4,\,2,\,0.5,\,0.3$.

2) 2D Burgers equation
\begin{eqnarray}
\frac{\partial u}{\partial t} + u\frac{\partial u}{\partial x} + v\frac{\partial u}{\partial y}
&=& \frac{1}{R_e}\Delta u,   \label{Eq4.5} \\
\frac{\partial v}{\partial t} + u\frac{\partial v}{\partial x} + v\frac{\partial v}{\partial y}
&=& \frac{1}{R_e}\Delta v,  \label{Eq4.6}\\
(x,y)\in\Omega &=& [0,1] \times [0,1],\,\,  t>0, \nonumber
\end{eqnarray}
where $R_e=80$ is the Reynolds number, the analytical solutions below at $t=0$ specifies the initial conditions, and the boundary conditions are also specified by the following analytical solutions.
\begin{eqnarray}\label{Eq4.5}
u(x,y,t)&=&\frac{3}{4}- \frac{1}{4[1+\exp((-t-4x+4y)R_e/32)]},  \label{Eq4.7} \\
v(x,y,t)&=&\frac{3}{4}+ \frac{1}{4[1+\exp((-t-4x+4y)R_e/32)]}.  \label{Eq4.8}
\end{eqnarray}

3) 2D FitzHugh-Nagumo reaction-diffusion equation
\begin{eqnarray}
\frac{\partial u}{\partial t} &=&
\gamma_u\Delta u + u - u^3 - v + \alpha,   \label{Eq4.9} \\
\frac{\partial v}{\partial t} &=&
\gamma_v\Delta v + \beta(u-v),    \label{Eq4.10} \\
(x,y)\in\Omega &=& [0,1] \times [0,1],\,\,  t>0, \nonumber
\end{eqnarray}
where $\gamma_u = \gamma_v =0.001, \alpha=0, \beta=1.$

subject to the boundary conditions
\begin{equation*}
u(x,y,t) = 0, \,\, v(x,y,t)=0, \,\,(x,y)\in \partial\Omega, t>0,
\end{equation*}

and the initial conditions
\begin{eqnarray}\label{Eq4.6}
u(x,y,0)&=&\sin(b\pi x)\cos[a\pi(y-d)]e^{-c(x^2+y^2)},  \label{Eq4.11} \\
v(x,y,0)&=&\cos[a\pi(x-d)]\sin(b\pi y)e^{-c(x^2+y^2)},  \label{Eq4.12}
\end{eqnarray}
where $a,\,b,\,c,\, d = 1,\,1,\,6,\,0$.

\end{document}